\documentclass[a4paper,11pt]{article}

\usepackage{amsmath,amsfonts,amssymb,mathtools,amsthm}
\usepackage{mathrsfs,enumerate}
\usepackage[normalem]{ ulem }
\usepackage{soul}

\usepackage{mathabx}
\usepackage[latin1]{inputenc}
\usepackage[dvipsnames,usenames]{color}
\usepackage{bbm}
\usepackage{dsfont}
\definecolor{bluegray}{rgb}{0.4, 0.6, 0.8}

\definecolor{bluered}{rgb}{0.6,0.3,0.4}

\def\bluegray{\color{bluegray}}

\usepackage{color}

\definecolor{orange}{rgb}{1,0.5,0}

\usepackage[text={152mm,225mm},centering]{geometry}

\usepackage[title,titletoc,toc]{appendix}

\usepackage[hyperindex=true, 
					colorlinks=false]{hyperref}

\numberwithin{equation}{section}

\newcommand{\begla}{\begin{equation}}
\newcommand{\beglab}[1]{\begin{equation}	\label{#1}}
\newcommand{\edla}{\end{equation}}

\definecolor{caputmortuum}{rgb}{0.35, 0.15, 0.13}
\definecolor{bulgarianrose}{rgb}{0.28, 0.02, 0.03}
\definecolor{charcoal}{rgb}{0.21, 0.27, 0.31}
\definecolor{coolblack}{rgb}{0.0, 0.18, 0.39}

\newcommand{\ie}{{\it i.e.}\ }
\newcommand{\cf}{{\it cf.}\ }
\newcommand{\eg}{{\it e.g.}\ }
\newcommand{\resp}{{resp.}\ }
\newcommand{\wrt}{{with respect to}}
\newcommand{\lhs}{{left-hand side}}
\newcommand{\rhs}{{right-hand side}}

\newcommand{\ens}{\enspace}

\def\dst{\displaystyle}

\newcommand{\ii}{^{-1}}
\newcommand{\ti}{\tilde}

\newcommand{\ex}{\mathrm{e}}

\newcommand{\dd}{\mathrm{d}}

\newcommand{\defeq}{\coloneqq} 
\newcommand{\col}{\colon\thinspace}          

\newcommand{\demi}{\frac{1}{2}}
\newcommand{\dem}{{\tfrac{1}{2}}}
\newcommand{\tiers}{\tfrac{1}{3}}
\newcommand{\dtiers}{\tfrac{2}{3}}
\newcommand{\quart}{\tfrac{1}{4}}
\newcommand{\bquart}{\tfrac{1}{4}}

\newcommand{\sspar}[1]{^{(#1)}}
\newcommand{\sz}{z_1\sspar{0}}
\newcommand{\ssz}[1]{z_2\sspar{#1}}
\newcommand{\ssij}{\sspar{i,j}}

\newcommand{\cc}[1]{_{[#1]}}

\newcommand{\br}{\bar r}

\def\hr{\hat r}
\def\sig{\sigma}
\def\th{\theta}

\newcommand{\brn}{\br\sspar n}
\newcommand{\brnp}{\br\sspar{n+1}}

\def\Z{{\mathbb Z}}\def\T{{\mathbb T}}\def\R{{\mathbb R}}

\def\cB{{\mathcal B}}\def\cG{{\mathcal G}}
\def\cT{{\mathcal T}}
\def\cN{{\mathcal N}}

\def\cS{{\mathcal S}}

\def\cV{{\mathcal V}}

\def\bz{\bar z}

\newcommand{\ga}{\gamma}

\newcommand{\al}{\alpha}

\def\th{\theta}

\def\L{\Lambda}

\def\<{\langle}
\def\>{\rangle}

\theoremstyle{plain}

\newtheorem{theo}{Theorem}[section]
\newtheorem{question}{Question}

\newtheorem*{Cor*}{Corollary}
\newtheorem*{Prop*}{Proposition}

\newcommand{\ph}{\varphi}

\def\ti{\tilde}

\def\pa{\partial}

\def\th{\theta}

\DeclarePairedDelimiter\absD{\lvert}{\rvert}%
\DeclarePairedDelimiter\normD{\lVert}{\rVert}%
\DeclarePairedDelimiter\angD{\langle}{\rangle}

\newcommand{\cF}{\mathcal{F}}

\newcommand{\cR}{\mathcal{R}}

\def\bm{\begin{matrix}}
\def\em{\end{matrix}}

\def\cV{{\mathcal V}}

\def\0{{\mathbf 0}}

\newtheorem{thm}{Theorem}[section]
\newtheorem{cor}[thm]{Corollary}

\newtheorem{lemma}[thm]{Lemma}

\newtheorem{prop}[thm]{Proposition}

\theoremstyle{remark}
\newtheorem{rem}{Remark}[section]

\def\cB{\mathcal{B}}

\theoremstyle{definition}

\newtheorem{definition}{Definition}[section]

\def\ssk{\smallskip}
\def\msk{\medskip}

\def\ssm{\smallsetminus}

\renewcommand{\setminus}{\ssm}

\newcommand{\Dist}{\operatorname{dist}}
\newcommand{\Leb}{\operatorname{Leb}}

\newcommand{\supp}{\operatorname{supp}}

\newcommand{\eps}{{\epsilon}}
\newcommand{\De}{{\Delta}}

\newcommand{\de}{{\delta}}
\newcommand{\la}{{\lambda}}
\newcommand{\La}{{\Lambda}}

\newcommand{\Om}{{\Omega}}
\newcommand{\om}{{\omega}}

\newcommand{\N}{{\mathbb N}}
\newcommand{\Q}{{\mathbb Q}}

\def\B0{{\bold{0}}}

\newcommand{\hz}{{\hat z}}

\catcode`\@=12

\def\Empty{}
\newcommand\oplabel[1]{
  \def\OpArg{#1} \ifx \OpArg\Empty {} \else
  	\label{#1}
  \fi}
		
\newcommand{\ka}{\kappa}

\newcommand{\comm}[1]{}
\newcommand{\comment}[1]{}

\newcommand{\diff}{diffusive-$*$}
\newcommand{\epsc}{\eps_{\operatorname{c}}}
\newcommand{\epsd}{\eps_{\operatorname{d}}}
\newcommand{\epsi}{\eps_{\operatorname{i}}}
\newcommand{\epsf}{\eps_{\operatorname{f}}}
\newcommand{\epsH}{\eps_{\operatorname{H}}}

\newcommand{\per}{_{\operatorname{per}}}

\newcommand{\Id}{\operatorname{Id}}
\newcommand{\ID}{\operatorname{Id}}

\DeclarePairedDelimiter{\ceil}{\lceil}{\rceil}
\DeclarePairedDelimiter{\flo}{\lfloor}{\rfloor}

\makeatletter
\def\th@plain{%
  \thm@notefont{}
  \itshape 
}
\def\th@definition{%
  \thm@notefont{}
  \normalfont 
}
\makeatother

\newcommand{\IMP}{ \ens \Rightarrow \ens }
\newcommand{\LIP}{\operatorname{Lip}}

\newcommand{\modZ}{\!\! \mod \Z}

\newcommand{\noinbf}[1]{\bigskip

\noindent\textbf{#1)}}

\newcommand{\kG}{\cG^{\al,L}}

\newcommand{\ov}[1]{\overline{#1}}

\newcommand{\cK}{\mathcal{K}}

\begin{document}



\begin{center}

{\huge KAM tori are no more than sticky}



\vspace{.5cm}

B.~Fayad, D.~Sauzin

\vspace{.5cm}

6 December 2018

\end{center}

\medskip






\begin{abstract}
When a Gevrey smooth perturbation is applied to a quasi-convex
integrable Hamiltonian, it is known that the KAM invariant tori that
survive are 
``sticky'', \ie
doubly exponentially stable. We show by examples the optimality of 
this effective stability. 
\end{abstract} 


\bigskip

 \makeatletter
    \renewcommand{\l@section}{\@dottedtocline{1}{1.5em}{3em}}
    \renewcommand{\l@subsection}{\@dottedtocline{2}{3.0em}{3.5em}}
    \renewcommand{\l@subsubsection}{\@dottedtocline{3}{4.5em}{4.2em}}
    \makeatother
    
\noappendicestocpagenum
\setcounter{tocdepth}{1}
\tableofcontents


\bigskip

\section{Introduction} 



We are interested in effective stability around invariant
quasi-periodic tori of nearly integrable analytic or Gevrey regular
Hamiltonian systems. Under generic non degeneracy assumptions on the
integrable Hamiltonian, KAM theory (after Kolmogorv Arnold Moser)
guarantees the existence of a large measure set of invariant
quasi-periodic tori for the perturbed systems. The invariant tori
given by KAM theory have Diophantine frequency vectors.  To study the
diffusion rate of orbits that start near these invariant tori, an
important tool is the Birkhoff Normal Forms (or BNF) at an invariant
torus, that introduce action-angle coordinates in which the system in
small neighborhoods of an invariant Diophantine torus becomes
integrable up to arbitrary high degrees in the Taylor series of the
Hamiltonian (see for example \cite{Bi} or \cite{SM}).

Exploiting the Diophantine property of the frequency vector of the
invariant torus, it is possible to collect estimates in the successive
BNFs and establish exponential stability of the torus, in the sense
that nearby solutions remain close to the invariant torus for an
interval of time which is exponentially large with respect to some
power of the inverse of the distance $r$ to the torus, a power that
depends only on the Diophantine exponent $\tau$ of the torus in the
case of real analytic Hamiltonians, and that involves additionally the
degree of Gevrey smoothness in the case of Gevrey smooth Hamiltonians
(\cite{Wiggins}, \cite{Popov}). 
%

Combining BNF estimates with Nekhoroshev theory, Giorgilli and
Morbidelli proved in \cite{MG} that for integrable Hamiltonians with a
quasi-convex Hessian, the KAM tori of an analytic perturbation of the
Hamiltonian are doubly exponentially stable: the exponential stability
time $\exp\!\big(r^{-1/(\tau+1)}\big)$ is promoted to
$\exp\!\big(\!\exp\big(r^{-1/(\tau+1)}\big)\big)$. 
Invariant quasi-periodic tori with this strong form of effective
stability are termed {\it sticky}.

Stickiness of the invariant tori was later extended in \cite{BFN} to a
residual and prevalent set of integrable Hamiltonians and to the
Gevrey category. It was proved there that generically, both in a
topological and measure-theoretical sense, an invariant Lagrangian
Diophantine torus of a Hamiltonian system is doubly exponentially
stable. Also, for a residual and prevalent set of integrable
Hamiltonians, for any small perturbation in Gevrey class, there is a
set of almost full Lebesgue measure of KAM tori which are doubly
exponentially
stable. 

Our aim here, is to give examples showing that doubly exponential
stability cannot in general be
strengthened. 
Loosely stated our main result is the following 

\bigskip 

\noindent {\bf Theorem.}  {\it For arbitrary $N\ge3$, there exist
  quasi-convex Hamiltonian systems in 
  $N$~degrees of freedom that can be perturbed in the Gevrey smooth
  category so that most of the invariant tori of the perturbed system
  are no more than doubly exponentially stable.}

\bigskip 

The exact statements will be given in Section~\ref{sec:statem}.  The diffusion mechanism we will use in our constructions is the so called Herman synchronized diffusion, that first appeared in  \cite{hms} where the speed in Arnol'd diffusion is estimated for a class of nearly integrable system. In \cite{hms},  completely integrable systems with twist are considered  and it is shown that it is possible to construct perturbations of size $\eps$ in Gevrey class, that have orbits diffusing in action at an  exponential rate in inverse powers of $\eps$.  The diffusion rate is shown to be almost optimal due to the Nekhoroshev effective stability theory.  

Our setting here is quite different, in this that the perturbative
parameter is not an extra parameter $\eps$ but the {\it action
  variable itself} when viewed as the distance $r$ of the diffusive
orbit from the invariant torus.  In this ``singular perturbative
setting", the nature of the construction is in fact expected to be
different from \cite{hms} since the diffusion rate is at best doubly
exponential in an inverse power of~$r$, compared to 
%
%
the simple exponential one can achieve in
the Arnol'd diffusion problem treated in \cite{hms}.

The new difficulties that arise in the singular perturbative setting, as well as the novel constructions to overcome them  will be commented  in the next section where the heuristics of our construction are described in detail. 










\section{Description of the construction} \label{sec:heurist}


The construction of the diffusive flows is obtained by suspension from
a perturbative construction in the discrete setting of symplectic maps
on 
%
%
$M= (\T\times\R)^n \simeq \T^n\times\R^n$,
where $\T$ denotes the torus $\R/\Z$
and $n\ge2$.

We now explain the main ideas in the discrete construction in
  the case $n=2$.  We
concentrate on the diffusion rate from the neighborhood of a single
invariant torus.  We will be dealing with perturbation of a product of
two twist maps of the annulus $\T \times \R$, denote them by $F_0$ and
$G_0$, $F_0(\th_1, r_1)= (\th_1+\om_1+r_1+\Z, r_1)$,
$G_0(\th_2, r_2) \defeq (\th_2+\om_2+r_2+\Z, r_2)$. Set
$T_0 = F_0 \times G_0 \col M \righttoleftarrow$. Observe that $T_0$
has an invariant torus $\cT_0=\T^2\times \{(0,0)\}$, on which the restricted
dynamics is the translation of vector $\om=(\om_1,\om_2)$.

Let us explain how to perturb $T_0$ into a map $T$ that is tangent to
$T_0$ at $\cT_0$ and that has pieces of orbits that diffuse away from a
neighborhood of~$\cT_0$ at a doubly exponentially small speed. 
More precisely, we obtain a sequence $\rho_n \to 0$, points~$z_n$
such that $\Dist(z_n,\cT_0)<\rho_n$ and times $\Theta_n$ that are
doubly exponentially large in $1/\rho_n$, such that
$\Dist(T^{\Theta_n}(z_n),\cT_0)$ and $\Dist(T^{-\Theta_n}(z_n),\cT_0)$
are both doubly exponentially large in $1/\rho_n$.
It will appear clearly from our diffusion mechanism that drifting away
from $\cT_0$ by an amount $\rho_n$, or by an amount that is doubly
exponentially large in $1/\rho_n$, both require a doubly exponentially large time.


\subsection*{\bluegray  Herman synchronized diffusion}  


The diffusion mechanism we will use is the Herman synchronized
diffusion, that first appeared in \cite{hms}.  Let us explain in some
words what is the synchronized diffusion. It is based on the following
mechanism of coupling of two twist maps of the annulus (the second one
being integrable with linear twist): at exactly one point~$p$ of a
well chosen periodic orbit of period~$q$ of the first twist map 
in $M_1=\T\times\R$, 
the coupling consists of pushing the orbits in the second annulus up
in $M_2=\T\times\R$
on some fixed vertical $\Delta$ by an amount $1/q$ that sends an
invariant curve whose rotation number is a multiple of $1/q$ exactly
to another one having the same property (due to the linear twist
property).

The dynamics of the $q^{\rm th}$ iterate of the coupled map on the
line $\{p\}\times \Delta \subset M_1\times M_2$ will thus drift at a
linear speed : after $q^2$ iterates the point will have moved by~$1$
in the second action coordinate~$r_2$, and after $q^3$ it will have
moved by~$q$.
The diffusing orbits obtained this way are
  bi-asymptotic to infinity: their $r_2$-coordinates travel from~$-\infty$
  to~$+\infty$ at average speed $1/q^2$.


%
For this mechanism to be implemented with a Gevrey regular
small coupling of the two twist maps, it is necessary that the
periodic point $p$ be isolated from the rest of the points on its
orbit by a distance $\sigma$ that is greater than the inverse of  some
power of $\ln q$, since $1/q$ is the translation amount required from
the coupling that must be exclusively localized  around $p$. We call
such periodic points ``logarithmically'' isolated. 



\subsection*{\bluegray  Optimal rates in Arnol'd diffusion of \cite{hms}}  


In \cite{hms}, only one periodic point is sufficient to have estimates on the Arnol'd diffusion  rate in the nearly integrable system. In fact, in \cite{hms}, a completely integrable twist map  of the annulus such as $F_0$ is first perturbed to create a hyperbolic saddle point with a saddle connection (a pendulum). Near the separatrix of the pendulum,  
one can find periodic orbits of arbitrary high period $q$ and isolation $\sigma$ that is determined by the hyperbolicity  of the saddle point. More precisely, with an $\eps$ perturbation of the integrable twist, the periodic orbits near the separatrix will then have an isolation of order $\eps^{1/2}$ and choosing $q$ exponentially large in the inverse of $\eps$ allows to use the coupling mechanism with a second completely integrable linear twist to obtain diffusive orbits at exponential rate in the inverse of $\eps$. 


\subsection*{\bluegray  Doubly exponential diffusion rates in the singular perturbative setting}  


In our singular perturbative setting, the main obstacles when one attempts to apply the synchronized diffusion mechanism are  threefold : 1) the diffusion rate must be calculated from arbitrary small neighborhoods of the invariant torus $\cT_0$, hence many perturbations and many diffusive orbits may enter into play as opposed to the single orbit of \cite{hms}; 2) each perturbation must not affect $\cT_0$ and must allow for further perturbations; 3) the Diophantine property on the frequency of $\cT_0$ imposes, due to averaging, strong restrictions on the period and the isolation properties of the periodic points that come near the invariant torus.

The main step to prepare for the coupling construction is to be able
to perturb~$F_0$ in order to get an annulus map~$F$ 
 on the first factor $M_1=\T\times\R$
 that is tangent to~$F_0$ at the circle $r=0$ (we omit the
 subscript~$1$ for $r_1$ in this 
paragraph) and that
 has a sequence of periodic points $p_n$ at distance $r_n$ from the
 circle $r=0$ and that are~$\sigma_n$ isolated from their
 orbits. Since we will work with perturbations of~$F_0$ that are
 compactly supported away from $r=0$, we cannot expect larger
 isolation~$\sig_n$ than an exponentially small quantity in the
 inverse of~$r_n$, 
and the precise exponent involved in this
   exponential is dictated by the Gevrey regularity~$\al$ only.
%
%
%
   According to the above description of how the synchronized
   diffusion mechanism functions, the period~$q_n$ of the point~$p_n$,
   that will also determine the order of the diffusion rate, should
   not be taken smaller than an exponential in the inverse
   of~$\sigma_n$. Hence, the double exponential in $\frac{1}{r_n}$~!!

Let us now suppose that a map~$F$ is constructed with such a sequence
$p_n\in M_1$, and let us show how to obtain the coupling with the
map~$G_0$ which lives on the other factor
  $M_2=\T\times\R$. The main idea is to couple~$F$ and~$G_0$
separately at each periodic orbit with compactly supported Gevrey
regular coupling functions. Indeed, while performing locally the
couplings around the product of the orbit of $p_n$ with~$M_2$,
%
%
we keep the direct product structure of $F$ with $G_0$ in the products
of smaller neighborhoods of the circle $\T\times\{0\} \subset M_1$
with~$M_2$. Thus, the couplings that yield diffusive orbits involving
the successive points~$p_n$ are done inductively without affecting
each other.


\subsection*{\bluegray How to perturb the first factor to get a sequence of isolated periodic points}  


We turn now to the perturbative techniques that allow to
obtain~$F$. We put together two tools.
The first one allows us to perturb a periodic circular rotation of
period $P/Q$ while fine-tuning the rotation number so as to create circle diffeomorphisms with  a
$\sigma$-isolated periodic point $p$ of arbitrary large period~$q$,
where $\sigma$ is 
exponentially small in~$Q$ for large~$Q$ but otherwise independent
of~$q$.
%
%
In particular, we can choose~$q$ 
exponentially large in $1/\sigma$ (%
%
doubly exponentially large  in~$Q$). 
The second tool is a trick due to M.~Herman that
allows us to embed the circle dynamics thus obtained inside the phase
portrait of a perturbation of a linear twist map of the
annulus~$M_1$. In fact, the periodic map will appear in the
neighborhood of the circle of period $P/Q$ of the linear twist.
The coupling mechanism will then yield orbits that diffuse at speed
$1/q^2$ in $M_1\times M_2$. Of course, to conclude we
have to require that $1/Q$ be larger than the distance
$r_Q=\absD{\om_1-P/Q}$
from the circle $\T\times\{0\}$ to the periodic orbit of the linear twist near which the
isolated periodic point~$p$ is embedded. 
%
%
At that stage, we could assume~$\om_1$ irrational or Diophantine and
then impose that $1/Q$ be of order $\sqrt{r_Q}$ or larger, depending
on the Diophantine exponent of~$\om_1$, 
  with the hope to refine the estimates on~$\sig$ and thus on the
  diffusing time.
%
%
But, since we want to embed, using Herman's trick, the isolated
periodic point in~$M_1$ at distance~$r$ {\it without affecting} the circle
$\T\times\{0\}$, we must accept 
%
%
  the exponential smallness of the isolation parameter in~$M_1$ to be
  dictated in the first place by the Gevrey regularity of our compactly supported perturbations, thus absorbing
  the potential gain stemming from arithmetics.
This means that by our technique we cannot tackle the problem of
matching the diffusion rate with the doubly exponential stability
lower bounds obtained in \cite{MG,BFN}, that are of the
form 
$\exp\!\big(\!\exp\big(r^{-\frac{1}{\alpha(\tau+1)}}\big)\big)$,
where~$\alpha$ refers to the Gevrey regularity class and~$\tau$ is the
Diophantine exponent of the translation vector.

To emphasize the role that arithmetics should play in optimizing the diffusion speed we may ask the following question that is similar to the one raised in  \cite[Question 24]{FKicm} for elliptic fixed points. 

\begin{question} Give an example of an analytic or Gevrey smooth Hamiltonian that has a non-resonant invariant torus with positive definite twist that is not more than exponentially stable in time. \end{question}

It follows from \cite{MG,BFN,BFN_point}, that a super Liouville property must be required on the frequency vector of the invariant torus.
 
\bigskip

\subsection*{\bluegray Questions on the optimality of the bounds, on
  analytic perturbations and on the genericity of doubly exponential
  diffusion}


An interesting way to address the question of optimizing the bounds,
as well as to aim at analytic constructions, is to look for a single
analytic (or Gevrey smooth) perturbation of $F_0$ that yields a map
$\tilde{F}$ that is tangent to $F_0$ at $0$ to some fixed degree, and
that has a sequence of periodic orbits $p_n$ with isolation properties
related to the arithmetics of $\om_1$ (for instance, with $\sigma_n$
of order $e^{r_n^{-1/(\tau'+1)}}$, where $\tau'$ is such that $\om_1$
is not $\tau'$-Diophantine).  But even if such a perturbation of $F_0$
is possible, it would still be a delicate task to perform an analytic
coupling with $G_0$ since a single analytic intervention to couple the
neighborhood of the orbit of any periodic point $p_n$ with the second
factor will affect the whole map everywhere and we cannot rely on the
nice direct product structure of $F$ with $G_0$ for further
perturbations as we do in the Gevrey category. One should probably
resort to the theory of normally hyperbolic invariant manifolds to say
that some kind of product structure remains valid at the periodic
orbits of the points $p_n$. Even then however, the linear character of the  twist of the
second factor will definitely disappear, which will also bring extra
difficulties.

Let us make a last remark concerning the analytic category. In fact, obtaining examples of analytic Hamiltonians having a topologically unstable invariant torus with positive definite twist at the torus is a hard task by itself, let alone the control of the diffusion time that is the object of our investigation here. Real analytic Hamiltonians with unstable invariant tori and elliptic fixed points (with arbitrary frequencies in the case of $4$ degrees of freedom) were obtained in \cite{F,FF}, but these examples do not have positive  definite  twist.

Finally, besides the analytic question and the question of optimizing the bounds,  one can ask whether the upper bounds on the diffusion rates that we obtain in our examples are generic for KAM tori, or for invariant quasi-periodic Diophantine tori in general.


\subsection*{\bluegray Plan of the paper}  


Section~\ref{sec:statem} contains the main statements for
symplectomorphisms and for flows.
In Section~4, we state the main inductive step of the construction,
that yields a diffusive segment of orbit for a perturbation of
$F_0 \times G_0$ linked to one isolated periodic point that will be
created on the first factor. 
%
%
In Section~4.1 we explain how the main
inductive step is used to result in a diffusive invariant torus. In
Section~4.2 we elaborate on this to get simultaneously a large measure
set of invariant tori that are diffusive. 

Sections~5 and~6 contain the
proof of the main inductive step.
Section~5 is devoted to the perturbation of~$F_0$ in order to get the
map~$F$ with isolated periodic points. In Section~5.1, we show how to
perturb circular rotations to obtain a periodic orbit with the
required isolation estimate and with arbitrary large period. In
Section~5.2 we show how this periodic orbit can be imbedded in a
perturbation of~$F_0$.  Section 6 introduces the coupling lemma of~$F$
with~$G_0$ and shows how to use it to get diffusion using the isolated
periodic point of~$F$.

In Section~7 we provide the suspension trick that allows to transfer the results from the discrete case of symplectomorphisms to the continuous time context of Hamiltonian flows. 

In the Appendix we collect and prove some necessary Gevrey estimates for maps and for flows that are used all along the paper. 


\section{Statements}   \label{sec:statem}


\subsection{Notations on diffusive dynamics and Gevrey functions} 


We use the notation
%
%
\beglab{eqdefE}
E(\nu) = E_{C,\ga}(\nu) \defeq 
\ex^{
\ex^{ C \nu^{-\ga} } 
%
%
}
\quad \text{for $\nu>0$,}
%
\edla
for some choice of $C,\ga>0$ 
that we will explicit later.  
%
%
We will say informally that ``$E(\nu)$ is doubly exponentially large for small~$\nu$''.
Notice that
\beglab{ineqEElaN}
E(\nu) \gg E(\la\nu)^\mu \ens \text{as $\nu\to0$,}
\qquad \text{for every $\la>1$ and $\mu>0$.}
\edla
 
\begin{definition}    \label{defDiffusive}
Given a transformation~$T$ (or a flow) on a metric space $(M,d)$ and $\nu>0$, we say that:
\begin{itemize} 
\item
A point~$z$ of~$M$ is \textit{$\nu$-diffusive}
%
%
if there exist an initial condition $\hz \in M$ and 
%
%
a positive integer (or real)~$t$
such that $d(\hz,z)\le\nu$, 
$t\le E(\nu)$ 
%
%
and $d(T^t \hz,z)\ge E(2\nu)$.


\item
A subset~$X$ of~$M$ is \textit{$\nu$-diffusive} if all points in~$X$ are
$\nu$-diffusive.

\item
A subset~$X$ of~$M$ is \textit{diffusive} if there exists a sequence $\nu_n \to
0$ such that 
$X$ is $\nu_n$-diffusive for each~$n$.
%
%
\end{itemize}
In the latter cases, we also say that $T$ is \textit{$\nu$-diffusive
  on~$X$}, \resp \textit{diffusive on~$X$}.
\end{definition}



Our goal is to construct examples of diffusive dynamics in
the context of near-integrable Hamiltonian systems and
exact-symplectic maps.
The requirement $d(T^t \hz,z)\ge E(2\nu)$ might seem
  exaggerated at first look, but, as mentioned earlier, there will be
  no essential difference in the order of magnitude of the time needed
  to diffuse by an amount~$\nu$ or by an amount as large as
  $E(2\nu)$.


We will also use a variant of the above definition: we say that a
point or a set is \textit{$\nu$-\diff} or \textit{\diff} if the corresponding property
holds with the function~$E$ replaced by
\beglab{eqdefEstar}
E^*(\nu) \defeq E\big( \nu/\absD{\ln\nu} \big) = 
\exp\big(
\nu^{ \ga C \nu^{-\ga} }
\big)
%
\quad \text{for $\nu\in(0,1)$.}
\edla
Notice that, as $\nu\to0$, $E(\nu) \ll E^*(\nu) \ll \ex^{
\ex^{ C' \nu^{-\ga'} }}$ for any $\ga'>\ga$ and $C'>0$.

\msk

We will deal with Gevrey smooth functions and maps in several real
variables. 
Periodicity may be required \wrt\ some of these variables,
in which case we will consider that each of the corresponding variables
is an angle, which lives in
\[ \T\defeq \R/\Z. \]
Recall that, given a real $\al\ge1$, Gevrey-$\al$ regularity is
defined by the requirement that the partial derivatives exist at all
(multi)orders~$\ell$ and are bounded by $C M^{\absD\ell}
\absD\ell!^{\al}$ for some~$C$ and~$M$;
when $\al=1$ this simply means analyticity, but we shall take $\al>1$
throughout the article. 
Upon fixing a real $L>0$ which essentially stands for the inverse of
the previous~$M$, 
one can define a Banach algebra
$G^{\al,L}(K) \subset C^\infty(K)$
%
%
when~$K$ is a Cartesian product of closed Euclidean balls and tori;
the elements of $G^{\al,L}(K)$ are the ``uniformly
  Gevrey-$(\al,L)$'' functions on~$K$.
In the non-compact case of a Cartesian product $\R^N\times K$ with~$K$
as above, we define
\[ G^{\al,L}(\R^N\times K) \subset 
\kG(\R^N\times K) \subset 
C^\infty(\R^N\times K),
\]
where the smaller space is a Banach algebra, with norm $\normD{\,.\,}_{\al,L}$,
consisting of uniformly Gevrey-$(\al,L)$ functions on $\R^N\times K$,
while $\kG(\R^N\times K)$ is a complete metric space,
with translation-invariant distance $d_{\al,L}$,
%
obtained by covering~$\R^N$ by an increasing sequence of closed balls and considering
the Fr\'echet space structure accordingly.
Details are given in Appendix~\ref{AppSubsecGev}.
%


\subsection{Hamiltonian flows}   \label{sec:HamFlows}


Let $n \ge 2$.
We work in $\T^{n}\times \R^{n}$, with coordinates $(\th_1,\ldots,\th_n,r_1,\ldots,r_n)$,
or in $\T^{n+1}\times \R^{n+1}$, with coordinates
$(\th_1,\ldots,\th_n,\tau,r_1,\ldots,r_n,s)$.
We use the standard symplectic structures $\sum_{j=1}^{n} \dd\th_j\wedge\dd r_j$
or $\sum_{j=1}^{n} \dd\th_j\wedge\dd r_j + \dd\tau\wedge\dd s$,
%
%
so that it is equivalent to consider a non-autonomous
Hamiltonian $h(\th,r,t)$ on $\T^{n}\times \R^{n}$ which depends
$1$-periodically on the time~$t$
or a Hamiltonian of the form
$H(\th,\tau,r,s) = s + h(\th,r,\tau)$ on $\T^{n+1}\times \R^{n+1}$.
Given an arbitrary $\om\in\R^n$, we will be interested in
non-autonomous $1$-periodic perturbations of 
\beglab{eq:defhz} 
h_0(r) \defeq (\om,r)+\dem(r,r) 
\edla
or, equivalently, in certain autonomous perturbations of the integrable
Hamiltonian
\begla H_0(r,s) \defeq s+ h_0(r) \edla
for which we denote by $\cT_{(r,s)}$ the invariant torus $\T^{n+1}\times
\{(r,s)\}$ associated with any $r\in\R^n$ and $s\in\R$
%
(it carries the quasi-periodic motion $\dot\th=\om+r$, $\dot\tau=1$).

\begin{theo} \label{theorem.onetorus}   
Let $\al>1$ and $L>0$ be real. 
%
For any $\eps>0$ there is $h \in \kG(\T^{n}\times\R^n\times\T)$
such that
\begin{itemize} 
\item[(1)] $d_{\al,L}(h_0,h)<\eps$, 
\item[(2)] the Hamiltonian vector field generated by 
 $H \defeq s + h(\th,r,\tau)$ 
is complete and, for every $s\in\R$, the torus $\cT_{(0,s)} \subset \T^{n+1}\times\R^{n+1}$ is
invariant and diffusive for~$H$.
Here the exponent implied in~\eqref{eqdefE} is $\ga = \frac{1}{\al-1}$.
%
\end{itemize}
\end{theo}

Note that if $\om$ is Diophantine then, for any~$h$ satisfying~(1)
of Theorem~\ref{theorem.onetorus}, we know from \cite{MG,BFN} that
$\cT_{(0,s)}$ is doubly exponentially stable (because $H_0$ is quasi-convex). 
More precisely, it holds that for any initial condition that is at distance $\rho$ from $\cT_{(0,s)}$, the orbit will stay within distance $2\rho$ from  $\cT_{(0,s)}$ during time $\exp\!\big(\!\exp\big(r^{-\frac{1}{\alpha(\tau+1)}}\big)\big)$,
where~$\tau$ is the Diophantine exponent.
Theorem~\ref{theorem.onetorus} shows that we cannot expect in general a 
stability better than doubly exponential. Observe that we do not
recover however the factor $1/(1+\tau)$ in the exponent
  governing our lower bound on the
diffusion time.  

Note also that we know from \cite{MG} and \cite{BFN} that, for $H$
such that $d_{\al,L}(H_0,H)<\eps$, we have a set of invariant tori
that are doubly exponentially stable and fill a set whose complement
has measure going to~$0$ as $\eps \to 0$ (for~$r$ in the unit ball for
example).
Our next result gives an example where most of these tori are no more
than doubly exponentially stable.

\begin{theo} \label{theorem.manytori} 
Let $\al>1$ and $L>0$ be real. 
%
For any $\eps>0$, there exist $h \in \kG(\T^{n}\times\R^n\times\T)$
and a closed set $X_\eps\subset [0,1]$ with $\Leb(X_\eps)\ge 1-\eps$, such that
\begin{itemize} 
\item[(1)] $d_{\al,L}(h_0,h)<\eps$, 
\item[(2)] the Hamiltonian vector field generated by 
 $H \defeq s + h(\th,r,\tau)$ 
is complete and,
for each $r\in (X_\eps+\Z)\times\R^{n-1}$ and $s\in\R$, the torus
$\cT_{(r,s)}\subset \T^{n+1}\times\R^{n+1}$ is invariant and \diff\
for~$H$.
Here the exponent implied in~\eqref{eqdefEstar} is $\ga = \frac{1}{\al-1}$.
%
\end{itemize}
\end{theo}

Both theorems will be proved in Section~\ref{sec.flows} by suspension
of analogous results which deal with exact-symplectic map and
which we state in the next section.
In fact, the union $\T^{n+1} \times (X_\eps+\Z) \times \R^n$ of all the tori
mentioned in Theorem~\ref{theorem.manytori} will be shown to be itself
\diff\ with exponent $\ga = \frac{1}{\al-1}$. 

As mentioned in Section~\ref{sec:heurist}, the unstable orbits
  which we will construct to prove our diffusiveness statements are in
fact bi-asymptotic to infinity: 
we will see that their $r_2$-coordinates travel from~$-\infty$
to~$+\infty$.


\subsection{Exact-symplectic maps}     \label{secSymplMaps}





Let $\om=(\om_1,\om_2) \in \R^2$.
Recall that $\T=\R/\Z$.
We set $M_1 \defeq \T \times \R$ and $M_2 \defeq \T \times \R$ and
define $F_0 \col M_1 \righttoleftarrow$ and
$G_0 \col M_2 \righttoleftarrow$ by
\beglab{eq:defFzGz}
F_0(\th_1, r_1) \defeq (\th_1+\om_1+r_1+\Z, r_1), \quad
G_0(\th_2, r_2) \defeq (\th_2+\om_2+r_2+\Z, r_2), 
\edla
%
and we set
\begla
T_0 \defeq F_0 \times G_0 \col M_1 \times M_2 \righttoleftarrow
\edla
Using the identification $M_1\times M_2 \simeq \T^2\times \R^2$, we
call $\cT_0$ the torus $\T^2\times \{(0,0)\}$. 
This torus is invariant by $T_0$ and the restricted dynamics on it is the translation of vector $\om$.
More generally we set 
\[
\cT_{(r_1,r_2)} 
\defeq \T^2\times \{(r_1,r_2)\}
\quad \text{for any $(r_1,r_2) \in \R^2$.}
%
\]

We say that a function is \textit{flat} at a point or on a subset, if
it vanishes together with all its partial derivatives of all orders there.
Our first result for the discrete case is that one can find a Gevrey perturbation of the
integrable twist map~$T_0$ that is flat at the torus~$\cT_0$ (which
thus stays invariant) and for which this invariant torus is diffusive.

From now on,
when a function~$H$ on a symplectic manifold generates a complete
Hamiltonian vector field, we denote by $\Phi^{H}$ the time-$1$ map of the flow
(note that $t \mapsto \Phi^{tH}$ is then the continuous time flow
generated by~$H$).
Thus, endowing $M_1 \times M_2 \simeq \T^2\times \R^2$ with the
symplectic form $\dd\th_1\wedge\dd r_1 +\dd\th_2\wedge\dd r_2 $,
we can write
\beglab{eq:TzasPhihz}
T_0 = \Phi^{\om_1 r_1+\om_2 r_2 + \demi(r_1^2+r_2^2)}.
\edla

\begin{theo} \label{theorem.onetorus.discrete}   
Let $\al>1$ and $L>0$ be real.
For any $\eps>0$ there exist $u\in G^{\al,L}(M_1)$ and
$v\in G^{\al,L}(M_1\times M_2)$ such that
\begin{itemize} 
\item[(1)]   $u$ and~$v$ are flat for $r_1=0$, 
\item[(2)]  $\normD{u}_{\al,L}+\normD{v}_{\al,L} < \eps$,
\item[(3)] $\cT_{0}$ is invariant and diffusive for $T \defeq \Phi^v
  \circ \Phi^u  \circ T_0$,
with exponent $\ga=\frac{1}{\al-1}$ 
in~\eqref{eqdefE}.
\end{itemize}
\end{theo}
Here, when we write~$\Phi^u$ with a function $u\col M_1\to \R$, we
view~$u$ as defined on 
$M_1\times M_2$ (and independent of the
variables~$\th_2$ and~$r_2$) and thus mean 
$\Phi^u \col M_1\times M_2 \righttoleftarrow$.

Our next result is a strengthening of
Theorem~\ref{theorem.onetorus.discrete}, in which we find a
perturbation that keeps invariant most of the tori of~$T_0$ while
insuring that they do become diffusive.

We will see
that, since the construction of~$u$ and~$v$ is
completely local, one can insure that in addition they 
%
%
are $1$-periodic in~$r_1$. 
If now $X\subset [0,1]$ is a closed set and if $u$ and~$v$ are flat
for $r_1 \in X$, then all the tori of the form $\cT_{(r_1,r_2)}$,
$r_1 \in X+\Z$, 
that are invariant by~$T_0$, are also invariant by
$T= \Phi^v \circ \Phi^u \circ T_0$ and carry the same translation
dynamics of vector $\om+(r_1,r_2)$. 
If moreover $u$ and~$v$ are such that there are diffusive orbits
for~$T$ on sufficiently dense scales in the neighborhood of the tori
$\cT_{(r_1,r_2)}$, $r_1 \in X+\Z$, then all these tori will be diffusive.
This is the content of the next result, in which we also control the
measure of the complement of the invariant tori (in bounded regions).

\begin{theo} \label{theorem.manytori.discrete} 
Let $\al>1$ and $L>0$ be real.
For any $\eps>0$ there exist $u\in G^{\al,L}(M_1)$ and
$v\in G^{\al,L}(M_1\times M_2)$ that are $1$-periodic in $r_1$,
and a closed set $X_\eps \subset [0,1]$ with 
$\Leb(X_\eps)\ge 1-\eps$,
so that
\begin{itemize} 
\item[(1)]   $u$ and~$v$ are flat for $r_1 \in X_\eps+\Z$. 
\item[(2)]  $\normD{u}_{\al,L}+\normD{v}_{\al,L} < \eps$,
\item[(3)] for each $(r_1,r_2) \in (X_\eps+\Z)\times\R$, the torus
  $\cT_{(r_1,r_2)}$ is invariant and \diff\ for
  $T \defeq \Phi^v \circ \Phi^u \circ T_0$,
with exponent $\ga=\frac{1}{\al-1}$ in~\eqref{eqdefEstar}.
\end{itemize}
\end{theo}

In fact, in~(3), the union $\T\times(X_\eps+\Z)\times M_2$ of all
these tori will be shown to be itself \diff\ for~$T$ with exponent $\ga=\frac{1}{\al-1}$ .


\begin{rem}   \label{rem:discretemultidim}
  As immediate corollaries, we get multidimensional versions of
  Theorems~\ref{theorem.onetorus.discrete}
  and~\ref{theorem.manytori.discrete}, in $\T^n\times \R^n$ with any
  $n\ge 3$, simply by taking direct product of the previous discrete
  systems with factors of the form $\Phi^{\om_i r_i + \demi r_i^2}$,
  $i\ge3$:
identifying $\T^n\times\R^n$ with $M_1\times\cdots\times
M_n$, where $M_i \defeq \T\times \R$ for each~$i$,
and setting
\[
T_0 \defeq \Phi^{h_0} \col \T^n\times\R^n \righttoleftarrow
\]
(with the same~$h_0$ as in~\eqref{eq:defhz}---this is thus a
generalization of~\eqref{eq:TzasPhihz}),
the statements of Theorems~\ref{theorem.onetorus.discrete}
and~\ref{theorem.manytori.discrete} hold \textit{verbatim} with this
new interpretation of~$T_0$ except that, in condition~(3) of
Theorem~\ref{theorem.manytori.discrete}, $\cT_{(r_1,r_2)}$ is to be
replaced with $\T^n\times\{r\}$ for arbitrary
$r\in (X_\eps+\Z)\times\R^n$
  and the functions~$u$ and~$v$ are to be viewed as functions on
  $\T^n\times\R^n$.
\end{rem}


\begin{rem}
To prove the above discrete-time diffusiveness statements, we will
exhibit orbits $(T^k\hz)_{k\in\Z}$ which satisfy the first requirement
of Definition~\ref{defDiffusive}, 
$d(T^t \hz,z)\ge E(2\nu)$ with a certain positive integer $t\le E(\nu)$,
for smaller and smaller positive values of~$\nu$.
We will see that, in fact, they even satisfy
$d(T^{j t} \hz,z)\ge \absD{j} E_{C,\ga}(2\nu)$ for all $j\in\Z$
and are thus bi-asymptotic to infinity, and more precisely their
$r_2$-coordinates grow linearly by an exact amount~$1/q$ after~$q$
iterates, where $q=t^{1/3}$ is integer.
\end{rem}

 
\section{The main building brick : localized diffusive orbits} 
\label{sec:mainmldbr}


We now fix real numbers $\al>1$ and $L>0$ once for all.
We also fix $\om\in\R^2$ and work with
$T_0 = \Phi^{\om_1 r_1+\om_2 r_2 + \demi(r_1^2+r_2^2)} \col M_1\times M_2 \righttoleftarrow$
as in Section~\ref{secSymplMaps}.

To prove Theorems~\ref{theorem.onetorus.discrete}
and~\ref{theorem.manytori.discrete}, and then the continuous time
versions of these, we will use the following building brick,
where we use the notation
\[
\cV(r,\nu) \defeq \T \times (r-\nu,r+\nu) \subset \T\times \R
\quad \text{for any $r \in \R$ and $\nu >0$.}
\]


\begin{prop} \label{prop.diffusion.discrete}
%
%
{Let $\ga \defeq \frac{1}{\al-1}$ and $b\defeq \bquart$. 
There exists $c = c(\al,L)$ 
such that},
for any $\nu>0$ small enough and $\br \in \R$,
there exist $u\in G^{\al,L}(M_1)$ and $v\in G^{\al,L}(M_1\times M_2)$ such that
\begin{enumerate}[(1)]
\item   $u\equiv 0$ on $\cV(\br,\nu)^c$ and $v\equiv 0$ on $\cV(\br,\nu)^c \times M_2$,
\item  $\normD{u}_{\al,L}+\normD{v}_{\al,L} 
%
%
{\le \ex^{-c \nu^{-\ga}}}$,
%
%
\item the set $\cV(\br,\nu) \times M_2$ is invariant and $\big(3\nu,\tau,\tau^{b}\big)$-diffusive 
%
%
for $T \defeq \Phi^v \circ \Phi^u \circ T_0$,
{where 
$\tau \defeq E_{3c\ga,\ga}(\nu)$}. 
%
%
\end{enumerate}
\end{prop}


In Condition~(3) of the statement, we have used a refinement of
Definition~\ref{defDiffusive}: 
%
%
we say that a subset~$X$ of~$M$ is \textit{$(\ti\nu,\tau,A)$-diffusive
  for~$T$} if, for every point~$z$ of~$X$, there exist 
$\hz \in M$ and~$t$ integer such that 
$d(\hz,z)\le \ti\nu$, $t\le \tau$ 
and $d(T^t \hz,z)\ge A$.



The proof of Proposition~\ref{prop.diffusion.discrete} is in
Sections~\ref{secIsolPerPts} and~\ref{secSynchrDiffH}.
(The reader will see that our choice of $b=\bquart$ is quite arbitrary: any positive
number less than~$\tiers$ would do.)


%
\subsection{Proof that Proposition~\ref{prop.diffusion.discrete}
  implies Theorem~\ref{theorem.onetorus.discrete}}


We are given $\eps>0$ and, without loss of generality, we can assume
that~$\eps$ is small enough so that we can apply
Proposition~\ref{prop.diffusion.discrete} for every $n\ge1$ with the
following values of~$\nu$ and~$\br$:
\[
\nu_n \defeq (c\ga)^{1/\ga} 10^{-n} \eps,
\quad
\brn \defeq 2 \nu_n.
\]
We thus obtain $u_n\in G^{\al,L}(M_1)$, supported in $\cV(\brn,\nu_n)$,
and $v_n\in G^{\al,L}(M_1\times M_2)$, supported in $\cV(\brn,\nu_n) \times M_2$,
which also satisfy Conditions~(2) and~(3) of Proposition~\ref{prop.diffusion.discrete}.
Observe that for any two different values of~$n$ the supports are
disjoint
(because $\brn-\nu_n > \brnp+\nu_{n+1}$).

Since 
%
%
\beglab{ineqexp}
\ex^{x} > x
\quad\text{and}\quad
\ex^{-x} < x\ii \quad \text{for all $x>0$,}
%
\edla
%
we have 
%
\beglab{ineqexcnuga}
\ex^{-c \nu_n^{-\ga}} = \big(\ex^{-c\ga \nu_n^{-\ga}}\big)^{1/\ga} 
< \big(c\ga \nu_n^{-\ga}\big)^{-1/\ga} = (c\ga)^{-1/\ga} \nu_n,
%
\edla
hence the formulas
$u \defeq \sum_{n=1}^\infty u_n$ and $v\defeq \sum_{n=1}^\infty v_n$
define functions $u\in G^{\al,L}(M_1)$ and $v\in G^{\al,L}(M_1\times M_2)$
with $\normD{u}_{\al,L}+\normD{v}_{\al,L} \le \sum \ex^{-c \nu_n^{-\ga}}
< \eps$.
Moreover, since the~$u_n$'s, $v_n$'s and all their partial derivatives
vanish for $r_1=0$, the same is true for~$u$ and~$v$.
The functions~$u$ and~$v$ thus satisfy properties~(1)--(2) of Theorem
\ref{theorem.onetorus.discrete}.

We claim that $T= \Phi^{v} \circ \Phi^u \circ T_0$ also satisfies
property~(3) of Theorem \ref{theorem.onetorus.discrete}.
Indeed, the disjointness of the supports implies that~$T$ coincides
with $\Phi^{v_n} \circ \Phi^{u_n}\circ T_0$ on
$\cV(\brn,\nu_n) \times M_2$.
Now, for each $z\in\cT_0$ and $n\ge1$, we can pick $\bz \in
\cV(\brn,\nu_n)\times M_2$ such that $d(\bz,z) < 2\nu_n$,
%
%
and then, by~(3) of Proposition~\ref{prop.diffusion.discrete}, we can find
$\hz \in \cV(\brn,\nu_n)\times M_2$ and 
  $t\le  \tau_n \defeq E_{3c\ga,\ga}(\nu_n)$ 
such that $d(\hz,\bz)\le 3\nu_n$ and
{$d(T^t \hz,\bz) \ge \tau_n^{b}$}. 
%
%
We get $d(\hz,z) \le 5\nu_n$ and $\tau_n = E_{C,\ga}(5\nu_n)$ if we use $C\defeq
3c\ga\cdot 5^\ga$ in the definition~\eqref{eqdefE} of the doubly exponentially large function~$E_{C,\ga}$.
Then, $d(T^t \hz,z) > \tau_n^{b} - 2\nu_n > \demi
\tau_n^{b} \gg E_{C,\ga}(10\nu_n)$ by~\eqref{ineqEElaN},
hence~$T$ is $5\nu_n$-diffusive on~$\cT_0$ for every~$n$ large enough.
%
%
%
%

\subsection{Proof that Proposition~\ref{prop.diffusion.discrete} implies
Theorem~\ref{theorem.manytori.discrete}}
\label{secProofNi}

%
%
We are given $\eps>0$ and, without loss of generality, we can assume
$\eps\le1$.
Let $\ga \defeq \frac{1}{\al-1}$ and {let~$c$ 
be as in Proposition~\ref{prop.diffusion.discrete}}.


Here is a definition that we will use from now on:
we say that a subset~$Y$ of a metric space~$X$ is \textit{$\nu$-dense} if,
for every $z\in X$, there exists $\bz\in Y$ such that $d(z,\bz)\le\nu$.


\noinbf{a}
We first define a fast increasing sequence of integers by
\beglab{eqdefNi}
N_1 \defeq \ceil*{\exp( {4\ka/\eps} )},
\qquad
N_i \defeq N_{i-1} \ceil*{ \exp \big( \exp \big( 
\ti C ( N_{i-1} \ln N_{i-1} )^\ga 
\big) \big) }
\quad\text{for $i\ge2$,}
\edla
{where $\ka\defeq \max\{1,(c\ga)^{-1/\ga}\}$ and $\ti C \defeq \max\{
6c\ga,1/\ga\}$}.
%
%
We also set
\[
\ti\nu_i \defeq \frac{1}{N_i}, \qquad
\nu_i \defeq \frac{\ti\nu_i}{\absD{\ln\ti\nu_i}} 
= \frac{1}{N_i \ln N_i}, \qquad
{\tau_i \defeq E_{3c\ga,\ga}(\nu_i).} 
\]
According to Lemma~\ref{lemineqNi} in the appendix, one has
\beglab{ineqdeuxnuiNi}
2\nu_i N_i = \tfrac{2}{\ln N_i} \le 2^{-2i+1} \eps/\ka < 1.
\edla


\noinbf{b}
We now construct a sequence $(Y_i)_{i\ge1}$ of mutually disjoint
subsets of $\R/\Z$ such that, for each $i\ge1$,
\begin{enumerate}[(i)]
\item
$Y_i$ is a disjoint union of at most~$N_i$ open arcs~$Y\ssij$,
\item
each of these open arcs can be written 
$Y\ssij = (r\ssij-\nu_i,r\ssij+\nu_i) \modZ$, with $r\ssij \in [0,1)$,
\item
$Y_i$ is $\ti\nu_i$-dense in $(\R/\Z) - \bigsqcup\limits_{1\le i'\le i-1} Y_{i'}$,
\item   \label{itemclosedarcs}
$(\R/\Z) - \bigsqcup\limits_{1\le i'\le i} Y_{i'}$ is a disjoint union
of~$N_i$ closed arcs of equal length.
\end{enumerate}

To do so, we start with $r\sspar{1,j} \defeq \frac{j-1}{N_1}$ for
$j=1,\ldots, N_1$ and define the arcs~$Y\sspar{1,j}$ and the set~$Y_1$
by~(i)--(ii)
(the disjointness requirement results from $2\nu_1 N_1 < 1$). 
We then go on by induction and suppose that, for a given $i\ge1$,
(i)--(iv) hold for $i'=1,\ldots,i$.
As a consequence of~(i)--(ii) and~\eqref{ineqdeuxnuiNi}, we have
\beglab{ineqLebYi}
\Leb(Y_{i'}) \le 2\nu_{i'} N_{i'} \le 2^{-i'}\eps/\ka \le 2^{-i'}
\quad \text{for $i' = 1,\ldots,i$.}
\edla
We observe that $M_{i+1} \defeq \frac{N_{i+1}}{N_i}$ is an integer $\ge3$.
Inside each closed arc mentioned in~\eqref{itemclosedarcs}, we can
place $M_{i+1}-1$ disjoint open arcs of length $2\nu_{i+1}$ so that the
complement is made of $M_{i+1}$ closed intervals of equal length.
Indeed, on the one hand 
$2\nu_{i+1}(M_{i+1}-1) < 2\nu_{i+1}N_{i+1}/N_i < 2^{-i-1}/N_i$,
on the other hand, the common length of the closed arcs of~\eqref{itemclosedarcs} is
\[
\mu_i = 
\tfrac{1}{N_i} \Big( 1 - \sum_{i'=1}^i \Leb(Y_{i'}) \Big) > 2^{-i-1}/N_i
\]
by~\eqref{ineqLebYi}.
Labelling all the new open arcs as $Y\sspar{i+1,j}$, where~$j$ runs through a
set of $N_i(M_{i+1}-1) < N_{i+1}$ indices, and calling~$Y_{i+1}$ their
union, we get the desired properties
(note that $Y_{i+1}$ is $\frac{\mu_i}{M_{i+1}}$-dense in each closed arc
mentioned in~\eqref{itemclosedarcs},
and  $\frac{\mu_i}{M_{i+1}} < \frac{1}{N_iM_{i+1}} = \ti\nu_{i+1}$).


\noinbf{c}
Now, for each~$i$ and~$j$, we apply
Proposition~\ref{prop.diffusion.discrete} with $\br = r\ssij$ just
constructed and $\nu=\nu_i$
(assuming~$\eps$ small enough so that the $\nu_i$'s are small enough
to allow us to do so).
We obtain Gevrey functions~$u\ssij$ and~$v\ssij$ supported in
$\cV(r\ssij,\nu_i)$ and $\cV(r\ssij,\nu_i)\times M_2$,
with 
\[
\normD{u\ssij}_{\al,L}+\normD{v\ssij}_{\al,L} \le \xi_i,
\quad\text{where $\xi_i \defeq \ex^{-c\nu_i^{-\ga}} < \ka\nu_i$}
\]
(incorporating~\eqref{ineqexcnuga}), so that $\cV(r\ssij,\nu_i)\times M_2$ is invariant and 
{$\big(3\nu_i,\tau_i,\tau_i^{b}\big)$}-diffusive for
$\Phi^{v\ssij} \circ \Phi^{u\ssij}\circ T_0$.
We set $u\ssij\per(\th_1,r_1) \defeq \sum_{k\in\Z}u\ssij(\th_1,r_1+k)$
and $v\ssij\per(\th_1,r_1,\th_2,r_2) \defeq \sum_{k\in\Z}v\ssij(\th_1,r_1+k,\th_2,r_2)$
so as to get functions which are $1$-periodic in~$r_1$ and have the
same Gevrey norms.

Consider the finite sums $u_i \defeq \sum_j u\ssij\per$ and
$v_i \defeq \sum_j v\ssij\per$ for each $i\ge1$:
the disjointness of the supports implies that 
\beglab{eqdefTidiffus}
\text{$T_i \defeq \Phi^{v_i} \circ \Phi^{u_i}\circ T_0$
is $\big(3\nu_i,\tau_i,\tau_i^{b}\big)$-diffusive on 
$\T \times \ti Y_i\times M_2$,}
\edla
where~$\ti Y_i$ is the lift of~$Y_i$ in~$\R$.
Finally, let $u\defeq \sum_i u_i$ and $v\defeq \sum_i v_i$.
These are well-defined Gevrey functions which are $1$-periodic
in~$r_1$ because
\beglab{inequivi}
\normD{u_i}_{\al,L}+\normD{v_i}_{\al,L} 
\le N_i \xi_i
< N_i \ka \nu_i = \frac{\ka}{\ln N_i} \le 2^{-i-1}\eps
\quad\text{for each $i\ge1$}
\edla
(by~\eqref{ineqdeuxnuiNi}),
hence the series over~$i$ are convergent in Gevrey norm, with
$\normD{u}_{\al,L}+\normD{v}_{\al,L} < \eps$.


\noinbf{d}
We claim that $T \defeq \Phi^{v} \circ \Phi^{u}\circ T_0$ satisfies Theorem
\ref{theorem.manytori.discrete} with 
\[
X_\eps \defeq 
\Big( \R - \bigsqcup\limits_{i\ge1} \ti Y_i \Big) \cap (0,1).
\]
%
%
%
To prove this claim, firstly we note that
$\Leb\big([0,1]-X_\eps\big) = \sum_i \Leb(Y_i) \le \eps$
by~\eqref{ineqLebYi}, and $X_\eps$ is closed because~$\ti Y_1$
contains~$0$ and~$1$.

%
Secondly the functions~$u$ and~$v$ vanish with all their partial derivatives
for $r_1\in X_\eps$ because the functions~$u\ssij\per$ and~$v\ssij\per$ and
all their partial derivatives do.

%
Lastly, by virtue of the previous point, for each $(r_1,r_2) \in
(X_\eps+\Z)\times\R$ the torus $\cT_{(r_1,r_2)}$ is invariant for~$T$,
thus it only remains for us to show that 
the union $\T\times(X_\eps+\Z)\times M_2$ of these tori
is \diff\ with exponent $\ga=\frac{1}{\al-1}$.
In view of the definition of~$X_\eps$, it is sufficient to show that,
for~$i$ large enough,
\beglab{claimXidiffusst}
\text{$Z_i \defeq M_1\times M_2 - \bigsqcup\limits_{1\le i'\le i-1} \T\times\ti Y_{i'}
  \times M_2$ \  is $2\ti\nu_i$-\diff\ for~$T$ with exponent~$\ga$.}
\edla
We will prove this with the constant $C\defeq 3^{\ga+1}c\ga$ in the
definitions~\eqref{eqdefE} and~\eqref{eqdefEstar} of the 
functions $E=E_{C,\ga}$ and $E^*=E_{C,\ga}^*$,
as a consequence of the fact that, for~$i$ large enough,
\beglab{claimXidiffus}
\text{$Z_i$ \ is {$(2\ti\nu_i,\tau_i,\dem \tau_i^{b})$}-diffusive for~$T$.}
\edla
Indeed, $\tau_i = E_{3c\ga,\ga}(\nu_i) = E(3\nu_i) = E\big( \frac{3\ti\nu_i}{\absD{\ln \ti\nu_i}} \big)
\ll E\big( \frac{2\ti\nu_i}{\absD{\ln \ti\nu_i}-\ln 2} \big) = E^*(2\ti\nu_i)$
by~\eqref{ineqEElaN}, and $\dem \tau_i^{b} \gg
E\big( \frac{4\ti\nu_i}{\absD{\ln \ti\nu_i}-\ln 4} \big) = E^*(4\ti\nu_i)$
still by~\eqref{ineqEElaN},
hence~\eqref{claimXidiffus} implies~\eqref{claimXidiffusst}.


\noinbf{e}
To prove~\eqref{claimXidiffus}, let $z\in Z_i$. We set $W_i \defeq \T\times \ti Y_i\times M_2$.
By property~(iii) of the construction of~$Y_i$, we
can find $\bz \in W_i$ such that $d(\bz,z) \le \ti\nu_i$,
and~\eqref{eqdefTidiffus} then yields $\hz \in W_i$
and $t\le \tau_i$ such that $d(\hz,\bz) \le 3 \nu_i$
and $d(T_i^t \hz, \bz) \ge \tau_i^{b}$.
We have $\absD{\ln\ti\nu_i} = \ln N_i > 3$ by virtue
of~\eqref{ineqdeuxnuiNi}, hence $3 \nu_i < \ti\nu_i$ and
$d(\hz,z) \le 2\ti\nu_i$.
Since $u_1,\ldots,u_{i-1},v_1,\ldots,v_{i-1}$ and all their
partial derivatives vanish on~$W_i$, the restriction of~$T$ to~$W_i$
coincides with that of 
\beglab{deftiTi}
\ti T_i \defeq \Phi^{v_i+g_i} \circ \Phi^{u_i + f_i} \circ T_0,
\quad \text{where}\quad
f_i \defeq \sum_{i'>i} u_{i'}, \quad
g_i \defeq \sum_{i'>i} v_{i'}.
\edla
Comparing the definition~$T_i$ in~\eqref{eqdefTidiffus}
and that of~$\ti T_i$ in~\eqref{deftiTi}, we observe
that the first~$t$ points on the $T_i$-orbit of~$\hz$ are close to the
first~$t$ points on its $T$-orbit (which is the same as its
$\ti T_i$-orbit). More precisely, \eqref{inequivi} yields
\[
\normD{u_i}_{\al,L} + \normD{f_i}_{\al,L} +
\normD{v_i}_{\al,L} + \normD{g_i}_{\al,L} \le 
\sum_{i'\ge i} N_{i'}\xi_{i'} \le 2^{-i}\eps,
\]
hence we can use Corollary~\ref{lemCompFlows} from the appendix for~$i$ large enough, 
obtaining 
\[
d\big(T_i^t(\hz),\ti T_i^t(\hz) \big) \le 3^{\tau_i} C_1 \big(
\normD{f_i}_{\al,L} + \normD{g_i}_{\al,L} \big)
\]
from~\eqref{ineqDistortEstim}, where $C_1 = C_1(\al,L)$, but
\[
\normD{f_i}_{\al,L} + \normD{g_i}_{\al,L} \le
\sum_{i'\ge i+1} N_{i'}\xi_{i'} \le 2 N_{i+1}\xi_{i+1}
\le 2 \cdot 3^{-\tau_i}
\]
by the first part of~\eqref{inequivi} and \eqref{ineqgeomNixi}--\eqref{ineqexpNixi},
whence
$d(T_i^t\hz, \ti T_i^t\hz) \le 2C_1 \le \dem \tau_i^{b}- d(\bz,z)$ 
and thus $d(T^t\hz, z) = d(\ti T_i^t\hz, z) \ge \dem \tau_i^{b}$
for~$i$ large enough,
and we conclude that~$z$ is {$(2\ti\nu_i,\tau_i,\dem \tau_i^{b})$-diffusive for~$T$.}
The proof of~\eqref{claimXidiffus} is thus complete.

\section{Isolated periodic points for twist maps of the annulus}
\label{secIsolPerPts}


\begin{definition}    
  Given a real $\sig>0$ and a discrete dynamical system in a metric
  space, we say that a periodic point is \textit{$\sig$-isolated} if it lies at
  a distance $\ge \sig$ of the rest of its orbit.
\end{definition}    

The goal of this section is to prove the following statement, which
will be instrumental in the proof of
Proposition~\ref{prop.diffusion.discrete}, where a perturbation
of~$F_0$ is obtained that has an isolated periodic point in~$M_1$.


\begin{prop} \label{prop.isolation} 
Let $\ga \defeq\frac{1}{\al-1}$.
%
%
There exists $c=c(\al,L)$ with the following property:
if we are given real numbers~$\nu$ and~$\sig$ with
%
%
$\nu>0$ small enough and
$0 < \sig \le \exp(- 2 c \nu^{-\ga})$, 
then, for any integer $\ell \ge 6 / \nu$ and
for any $\br \in \R$, 
there exists $u\in G^{\al,L}(M_1)$ such that
\begin{enumerate}[(1)]
\item $u\equiv 0$ on $\cV(\br,\nu)^c$,
\item {$\normD{u}_{\al,L} \le \demi \exp(- c \nu^{-\ga})$,}
\item $F \defeq \Phi^u \circ F_0\col M_1 \righttoleftarrow$ has a {$\sig$-isolated}
  periodic point $z_1\in \cV(\br,3\nu/4)$ of period ${q \in [\ell,3\ell/\nu]}$,
%
%
\item {the set $\{ F^s(z_1) \mid s \in \N,\; {2}/{\nu} \le
    s \le {6}/{\nu} \}$ is $2\nu$-dense in $\cV(\br,\nu)$.}
\end{enumerate} 
\end{prop}

The interest of the statement is that, although~$\sig$ is required to
be exponentially small, it can be kept independent of~$\ell$, even
if we choose~$\ell$, and thus~$q$, doubly exponentially large.

The proof will start with the construction of a circle map with an
isolated periodic point.


\subsection{Circle diffeomorphisms with isolated periodic points }  
\label{secCirclDiffeoIsol}


We start by constructing a circle diffeomorphism with an isolated periodic point.
For any $Q\in\N^*$ we denote by $\De_Q$ a function in $C^\infty(\T)$ satisfying
\beglab{eqdefDeQ}
\text{$\De_Q$ is $\tfrac{1}{Q}$-periodic,} \quad 
0 \le \De_Q \le 1, \quad 
\De_Q(0) = 0, \quad
\text{$\De_Q \equiv 1$ on $\big[ \tfrac{1}{4Q}, \tfrac{3}{4Q} \big].$}
\edla
%
%
Since $\al>1$, every space $G^{\al,L'}(\T)$ contains such a function, and one can choose it so that
\beglab{ineq:normDeQ}
\normD{\De_Q}_{\al,L'} \le L'^\al\exp(\ti c \, Q^\ga)
\edla
with $\ti c=\ti c(\al,L')$ independent of~$Q$
according to Lemma~\ref{lembump} in the appendix
(take \eg $\De_Q(\th) \defeq \sum_{j=1}^Q \eta_{2Q}\big(
\frac{2j-1}{2Q} + \th
\big)$
and $\ti c \defeq L'^{-\al} + \frac{1}{\ga} + 2^\ga c_1(\al,L')$
with the notation of Lemma~\ref{lembump}, using the inequalities
$L'^{-\al}\le \ex^{L'^{-\al}}$ and $1\le Q \le \ex^{\frac{1}{\ga}Q^\ga}$).


\begin{prop}   \label{propA}
Let $Q\in\N^*$ and $P\in\Z$ be coprime and let $\sig \in \big(0, \frac{1}{\max(-\De_Q')}\big)$.
Then, for every $\ell \in\N^*$, 
there exist an integer $q = q_\ell(Q,P)$ 
and a real
$\de = \de_\ell(\sig,Q,P)$
such that
\begin{itemize}
\item 
$\ell \le q \le \ell Q$ and
$0 < \de \le \frac{1}{\ell Q}$,
\item 
the point $\frac{1}{2Q}+\Z$ is periodic of
period~$q$ and $\sig$-isolated for the circle map
\beglab{eqdeffsigde}
\th \in \T \mapsto
f_{\sig,\de}(\th) \defeq \th + \frac{P}{Q} + \de + \sig \De_Q(\th)
\modZ.
\edla
\end{itemize}
If moreover $\ell\ge 2Q$, then {the set $\{ f_{\sig,\de}^s(\frac{1}{2Q}+\Z) \mid s \in \N,\; Q\le
    s \le 2Q-1 \}$ is $\frac{1}{Q}$-dense in~$\T$.}
\end{prop}


\begin{proof} 
\noinbf{a}
We define $q=q_\ell(Q,P)$ by writing
\[
\frac{\ell P+1}{\ell Q} = \frac{p}{q}
\quad
\text{with $q\in\N^*$ and $p\in\Z$ coprime.}
\]
Since $1$ is the only common divisor
of~$\ell$ and $\ell P+1$, we must have
\beglab{eq:divisellPpun}
\ell P+1 = p D, \quad Q = q' D, \quad q = \ell q'
\edla
with $D\in\N^*$ and $p\wedge q'=1$,
hence $\ell \le q \le \ell Q$.

The condition $0 < \sig < \frac{1}{\max(-\De_Q')}$ ensures that
$1+\sig\De_Q'$ stays positive hence, for every $\de\in\R$, the formula
\[
F_{\sig,\de}(x) \defeq x + \frac{P}{Q} + \de + \sig \De_Q(x)
\]
defines an increasing diffeomorphism of~$\R$ such that
$F_{\sig,\de}(x+1) = F_{\sig,\de}(x) +1$.
Formula~\eqref{eqdeffsigde} then defines a diffeomorphism~$f_{\sig,\de}$ of~$\T$, a lift of which
is~$F_{\sig,\de}$.
%
%
We will tune~$\de$ so as to get the rotation
number of~$f_{\sig,\de}$ equal to $p/q$.

%
\noinbf{b}
To study the dynamics of~$F_{\sig,\de}$ and particularly the orbit of
$x_0\defeq \frac{1}{2Q}$, we perform the change of
variable $X=Qx$ and set $X_0\defeq \demi$ and
\begin{align*}
G_{\sig,\de}(X) &\defeq Q F_{\sig,\de}(X/Q) 
= X+ P + \de Q + \sig Q \De_Q(X/Q),\\[1ex]
\ti G_{\sig,\de}(X) &\defeq X + \de Q + \sig Q \De_Q(X/Q).
\end{align*}
Note that also~$\ti G_{\sig,\de}$ is an increasing diffeomorphism of~$\R$ for
each $\de\in\R$. 
For every $\ell\in\N^*$ we have
\[
\ti G_{\sig,0}^\ell(X_0) < X_0 + 1
= \ti G_{0,\frac{1}{\ell Q}}^\ell(X_0)
\le \ti G_{\sig,\frac{1}{\ell Q}}^\ell(X_0)
\]
(the left inequality holds because $X_0<1$ and~$1$ is a fixed
point of~$\ti G_{\sig,0}$;
the right inequality holds because $\sig Q\De_Q(x)\ge0$ for all~$x$).
Therefore, since $\de \mapsto \ti G_{\sig,\de}^\ell(X_0)$ is continuous and increasing, we can
define $\de\defeq\de_\ell(\sig,Q,P)$ as the unique solution of the
equation 
\[
\ti G_{\sig,\de}^\ell(X_0) = X_0 + 1
\]
and we know that $0 < \de_\ell(\sig,Q,P) \le \frac{1}{\ell Q}$.

%
\noinbf{c}
We now fix~$\de$ to be this value~$\de_\ell(\sig,Q,P)$ and check that it
satifies the desired properties.
First, notice that
\beglab{ineq:sigquartQ}
4\sig Q <  1
\edla
(because $\De_Q(3/4Q)=1$ and $\De_Q(1)=0$, hence the mean value theorem implies
$
{\max(-\De_Q')} 
\ge {4Q}$).
Let us denote the full orbits of $X_0=\demi$ under~$G_{\sig,\de}$ and~$\ti G_{\sig,\de}$
by 
\[
X_j \defeq G_{\sig,\de}^ j(X_0), \quad \ti X_j \defeq \ti G_{\sig,\de}^ j(X_0), \qquad j \in \Z.
\]
We have $X_0 = \ti X_0 < \ti X_1 < \cdots < \ti X_{\ell-1} 
< \ti X_\ell = X_0 + 1$. In fact,
\beglab{ineq:tiXjzeroell}
X_0 + \sig Q < \ti X_1 < \cdots < \ti X_{\ell-1} < X_0+1-\sig Q.
\edla
Indeed, $\De_Q(X_0/Q) = 1$ hence $\ti X_1 = X_0 + \sig Q + \de Q$, and
either
$\ti X_{\ell-1} \ge X_0+1-\quart$, in which case
$X_0+1 = \ti G_{\sig,\de}(\ti X_{\ell-1}) = \ti X_{\ell-1} + \sig Q +
\de Q > \ti X_{\ell-1} + \sig Q$,
or $\ti X_{\ell-1} < X_0+1-\quart < X_0+ 1 -\sig Q$ by~\eqref{ineq:sigquartQ}.

Since $\De_Q$ is $\frac{1}{Q}$-periodic, we have $\ti G_{\sig,\de}(X+1) = \ti G_{\sig,\de}(X)+1$
and the pattern~\eqref{ineq:tiXjzeroell} repeats $1$-periodically:
for every $m\in\Z$, $\ti X_{\ell m} = X_0 + m$ and
\[
X_0 + m + \sig Q < \ti X_{\ell m + s} < X_0 + m + 1 - \sig Q
\quad \text{for $s = 1, \ldots, \ell-1$.}
\]
Since $\De_Q$ is $\frac{1}{Q}$-periodic, we have $G_{\sig,\de}^j(X) =
\ti G_{\sig,\de}^j(X)+jP$ for every $j\in\Z$, hence $X_j = \ti X_j +
jP$ and,
for every $m\in\Z$, 
\begin{gather}
\label{eq:XellmmppD}
X_{\ell m} = X_0 + m(\ell P+1) = X_0 + m p D \\[1ex]
\label{eq:XellmrM}
X_0 + M_{m,s} + \sig Q < X_{\ell m + s} < X_0 + M_{m,s} + 1 - \sig Q
\quad \text{for $s = 1, \ldots, \ell-1$,}
\end{gather}
with $M_{m,s} \defeq m(\ell P+1) + sP$.

%
\noinbf{d}
Going back to the variable $x=X/Q$, 
we see that the orbit $(x_j)_{j\in\Z}$ of~$x_0$ under~$F_{\sig,\de}$ satisfies
\beglab{eqxlmlmr}
x_{\ell m}  = x_0 + \frac{m p}{q'}
\quad\text{and}\quad
x_0 + \frac{M_{m,s}}{Q} + \sig < x_{\ell m + s} < x_0 + \frac{M_{m,s} + 1}{Q} - \sig
\edla
for $m\in\Z$ and $1\le s<\ell$
(thanks to~\eqref{eq:divisellPpun} and
\eqref{eq:XellmmppD}--\eqref{eq:XellmrM}).
In particular, $x_q = x_0 + p$, hence it induces a $q$-periodic orbit
of type $p/q$ for~$f_{\sig,\de}$.
The $\sig$-isolation property amounts to 
\[
\Dist(x_j,x_0 + \Z) \ge \sig \quad \text{for $1\le j< q = \ell q'$.}
\]
This holds because either $j=\ell m$ with $1\le m<q'$
and the first part of~\eqref{eqxlmlmr} entails
$x_{\ell m}-x_0 \in \Q-\Z$ with
$\Dist(x_{\ell m}-x_0,\Z) \ge \frac{1}{q'} \ge \frac{1}{Q} > 4\sig$ by~\eqref{ineq:sigquartQ},
or $j=\ell m+s$ with $1\le s<\ell$ and the second part
of~\eqref{eqxlmlmr} yields
$x_j-x_0 \in (\frac{M_{m,s}}{Q}+\sig, \frac{M_{m,s} + 1}{Q}-\sig)$,
but $(\frac{M_{m,s}}{Q}, \frac{M_{m,s} + 1}{Q}) \cap \Z = \emptyset$,
hence $\Dist(x_j-x_0,\Z) > \sig$.

%
\noinbf{e}
We now suppose $\ell \ge 2Q$ and prove the $\frac{1}{Q}$-density statement.

If $P=0$, then $Q=1$ (because of the assumption $P\wedge Q=1$) and
there is nothing to prove. We thus suppose $P\neq0$.
Using the second part of~\eqref{eqxlmlmr} with
$1 \le s \le 2Q-1< \ell$ and $m=0$,
since $M_{0,s}= sP$, we get 
\beglab{encadrxs}
x_0 + \frac{sP}{Q} < x_s < x_0 + \frac{sP}{Q} + \frac{1}{Q}
\quad \text{for $s = Q, \ldots, 2Q-1$.}
\edla
Since $P\wedge Q=1$, the $Q$ arcs $\big[x_0 + \frac{sP}{Q}, x_0 +
\frac{sP}{Q} + \frac{1}{Q}\big) \modZ$ are mutually disjoint and
cover~$\T$;
each of them has length~$\frac{1}{Q}$ and, according to~\eqref{encadrxs}, contains a point of
$\{ x_s+\Z \mid s = Q, \ldots, 2Q-1 \} = \{
f_{\sig,\de}^s(x_0+\Z) \mid s = Q, \ldots, 2Q-1 \}$.
This set is thus $\frac{1}{Q}$-dense in~$\T$.
\end{proof}



\subsection{Herman imbedding trick of circle diffeomorphisms into twist maps} 


We will have to imbed the circle dynamics of Proposition~\ref{propA}
into the annulus {\it via} a perturbation of the twist map~$F_0$. For
this we will use the celebrated technique introduced by Herman in
\cite{Hasterisque} to imbed circle dynamics as restricted dynamics on
an invariant graph by a twist map.

Recall that trivial examples of symplectic maps of the annulus are
given by 
\beglab{eq:trivex}
\Phi^h(\th,r)=\big(\th+h'(r)+\Z,r\big), \qquad
\Phi^w(\th,r)=\big(\th,r-w'(\th)\big),
\edla
where the function $h=h(r)$ does not depend on the angle $\th\in\T$,
and the function $w=w(\th)$ does not depend on the action~$r$.
In particular,
\beglab{CasePhihF}
h(r) \equiv \om_1 r + \dem r^2 
\IMP
\Phi^h = F_0.
\edla


\begin{prop}   \label{propB}
Suppose that we are given a circle diffeomorphism of the form
\[
\th \in \T \mapsto f(\th) = \th + h'\big(\hr + \eps(\th)\big) \modZ,
\]
where $h=h(r)$ and $\eps=\eps(\th)$ are smooth functions and
$\hr \in \R$.
Then the equation
\beglab{eqdefwp}
-w' = \eps - \eps\circ f\ii,
\edla
determines a smooth function $w=w(\th)$ up to an additive constant,
and the annulus map
\begla
\Phi^w \circ \Phi^h \col (\th,r) \mapsto \big( \th + h'(r)+\Z, r - w'(\th+h'(r)) \big)
\edla
leaves invariant the graph $\big\{ \big( \th, \hr + \eps(\th) \big)
\mid \th \in \T\big\}$, with induced dynamics $\th \mapsto f(\th)$ on
it.
\end{prop}


\begin{proof}
Let us check that the \rhs\ of~\eqref{eqdefwp} has zero mean value:
\[
\angD{\eps\circ f\ii} = \int_\T \eps\big( f\ii\big(\ti\th\big) \big)
\,\dd\ti\th
= \int_\T \eps(\th) f'(\th) \,\dd\th
\quad\text{and}\quad
f'(\th) = 1 + \frac{\dd\,}{\dd\th} \big[ h'\big(\hr+\eps(\th)\big) \big],
\]
hence $\angD{\eps} - \angD{\eps\circ f\ii} = 
- \int_\T \eps(\th) \frac{\dd\,}{\dd\th} \big[
h'\big(\hr+\eps(\th)\big) \big] \,\dd\th
= \int_\T \eps'(\th) h'\big(\hr+\eps(\th)\big) \,\dd\th
= 0$,
since this is the mean value of the derivative of the periodic
function
$h\big(\hr+\eps(\th)\big)$. 
Consequently, any primitive of $\eps-\eps\circ f\ii$ induces a smooth
function on~$\T$.


Now, consider an arbitrary point
$(\th,r) = \big( \th, \hr + \eps(\th) \big)$
on the graph mentioned in the statement.
Its image by $\Phi^w \circ \Phi^h$ is
$(\th_1,r_1) \defeq \big( \th+h'(r)+\Z, r-w'(\th_1) \big)$.
We have 
\[
\th_1 = 
\th + h'\big( \hr + \eps(\th) \big) +\Z = f(\th)
\]
and $r_1 = \hr + \eps(\th) - w'(\th_1) = \hr + \eps\circ f\ii(\th_1) - w'(\th_1) = \hr+\eps(\th_1)$.
\end{proof}


Let us have a look at the solutions of~\eqref{eqdefwp}
in the case of Gevrey data, with~$h$ as in~\eqref{CasePhihF},
\ie $h'(r) \equiv \om_1 + r$:

\begin{lemma}   \label{lemSolGevEqw}
Let $L'>L$. 
Suppose 
\[ 
f(\th) = \th + \om_1 + \hr + \eps(\th) \modZ
\quad\text{for all $\th\in\T$,}
\quad\text{and}\quad
\eps = \de + \eps_*, 
\]
where $\hr,\de \in \R$ and $\eps_* \in G^{\al,L'}(\T)$ satisfies
$\normD{\eps_*}_{\al,L'} \le \epsi$
with $\epsi = \epsi(\al,L,L')$ as in Lemma~\ref{lemGevInv}.
Then~$f$ is a circle diffeomorphism and Equation~\eqref{eqdefwp} has a solution~$w$ such that
\begla
\normD{w}_{\al,L} \le (1+2 L^\al)\normD{\eps_*}_{\al,L'}.
\edla
\end{lemma}


\begin{proof}
We can write $f = (\ID+\om) \circ (\ID+\eps_*)$ with $\om \defeq \om_1+\hr+\de$.
By Lemma~\ref{lemGevInv}, we obtain that $\ID+\eps_*$ is a
diffeomorphism of~$\R$, which (because of periodicity) can be viewed
as the lift of a circle diffeomorphism.
Hence~$f$ is a circle diffeomorphism and the \rhs\ of~\eqref{eqdefwp}
is
\[
g \defeq \eps - \eps\circ f\ii =
\eps_* - \eps_*\circ f\ii =
\eps_* - \eps_*\circ (\ID+\eps_*)\ii \circ (\ID-\om).
\]
Lemma~\ref{lemGevInv} yields 
\label{secUseGevInv}
$\normD{\eps_*\circ (\ID+\eps_*)\ii}_{\al,L} \le
\normD{\eps_*}_{\al,L'}$, which easily implies
$\normD{g}_{\al,L} \le 2\normD{\eps_*}_{\al,L'}$.

Now, we already know that~$g$ has zero mean value, and a solution
to~\eqref{eqdefwp} can be defined by the formula
\[
w(\th+\Z) \equiv - \int_0^\th g(\th_1)\,\dd\th_1
\quad\text{for $\th \in (-\dem,\dem]$.}
\]
One has $\normD{w}_{C^0(\T)} \le \dem \normD{g}_{C^0(\T)}$ and, for
each $k\ge0$,
\[
\frac{L^{(k+1)\al}}{(k+1)!^\al} \normD{w^{(k+1)}}_{C^0(\T)} =
\frac{L^{(k+1)\al}}{(k+1)!^\al} \normD{g^{(k)}}_{C^0(\T)} 
\le L^\al \cdot \frac{L^{k\al}}{k!^\al} \normD{g^{(k)}}_{C^0(\T)},
\]
whence $\normD{w}_{\al,L} \le (\dem+L^\al) \normD{g}_{\al,L}$
and the conclusion follows.
\end{proof}


\subsection{
Proof of Proposition~\ref{prop.isolation}}


\noinbf{a}
We start with an elementary fact.

\begin{lemma}   \label{lemchoiceofQ}
Suppose $x$ and~$\nu$ are real, with $\nu>0$ small enough. 
Then there exist $Q\in\N^*$ and $P\in\Z$ coprime such that
$\absD{x-P/Q} < \nu/2$ and 
$1 < \frac{2}{\nu} < Q < \frac{3}{\nu}$.
\end{lemma}

\begin{proof}
For $\nu>0$ small enough, thanks to the Prime Number Theorem, we can
pick a prime number~$Q$ in $(2\nu\ii,3\nu\ii)$
(\eg $Q=p_k$ with
$k\defeq \flo*{{5\nu\ii}/{2\absD{\ln\nu}}}$,
where $p_k \sim k\ln k$ is the $k$th prime number).
The interval $(Qx-\frac{Q\nu}{2},Qx+\frac{Q\nu}{2})$ has length $>2$,
hence it contains at least two consecutive integers, one of which is
not a multiple of~$Q$ and can be taken as~$P$.
\end{proof}


We now define 
\[
L'\defeq 2L, \qquad
c \defeq \max\{ 3^{\ga+1} \ti c, 2^{\ga+1} c_1 \}
\]
with $\ti c=\ti c(\al,L')$ as in~\eqref{ineq:normDeQ}
and $c_1 = c_1(\al,L)$ as in Lemma~\ref{lembump},
and suppose that we are given $\nu,\sig,\ell,\br$ as in the
statement of Proposition~\ref{prop.isolation}, with~$\nu$ small enough
so as to be able to apply Lemma~\ref{lemchoiceofQ}.

Applying Lemma~\ref{lemchoiceofQ} with $x = \om_1+\br$,
we get a rational $P/Q$ such that $P\wedge Q=1$ and
\[
1 < 2/\nu < Q < 3/\nu
\quad \text{and} \quad
\frac{P}{Q} = \om_1 + \hr
\quad \text{with} \quad 
\absD{\hr-\br} < \nu/2.
\]


\noinbf{b}
Let us choose a function
$\De = \De_Q \in G^{\al,L'}(\T)$ satisfying~\eqref{eqdefDeQ}--\eqref{ineq:normDeQ} and apply
Proposition~\ref{propA}.
We can do so {since $2c \ge 3^\ga\ti c$}, hence
{$0 < \sig \le \ex^{- 2 c \nu^{-\ga}}
<  \ex^{- \ti c Q^{\ga}} 
\le \frac{1}{\max(-\De')}$}
by~\eqref{ineq:normDeQ}.
We get an integer~$q$ and a real~$\de$ satisfying
\[
{
q \in [\ell , \ell Q] \subset [\ell, 3 \ell / \nu ],
\qquad
0 < \de \le \frac{1}{\ell Q} < \frac{\nu}{2\ell} \le \frac{\nu^2}{12},
}
\]
so that $\frac{1}{2Q}+\Z$ is a $\sig$-isolated periodic point
of period~$q$ for the circle map~$f$ defined by
\[
\th \in \T \mapsto
f(\th) \defeq \th + \frac{P}{Q} + \de + \eps_*(\th)
\modZ
\qquad\text{with $\eps_* \defeq \sig \De$.}
\]
Moreover, since $\ell\ge 6/\nu \ge 2Q$ and $\frac{1}{Q} < \frac{\nu}{2}$,
the set $\{ f^s(\frac{1}{2Q}+\Z) \mid s \in \N,\; 2/\nu\le
    s \le 6/\nu \}$ is $\frac{\nu}{2}$-dense in~$\T$.


\noinbf{c}
We are in the situation of Lemma~\ref{lemSolGevEqw}, with
\[
{
\normD{\eps_*}_{\al,L'} \le \sig L'^\al \ex^{\ti c \, Q^\ga}
\le L'^\al \ex^{-(2c-3^\ga\ti c)\nu^{-\ga}}
\le L'^\al \ex^{-\frac{5c}{3} \nu^{-\ga}}
}
\]
which is less than $\epsi(\al,L,L')$ for~$\nu$ small enough,
thus Equation~\eqref{eqdefwp} with $\eps \defeq \de+\eps_*$ 
and $h' = \ID + \om_1$ has a
solution $w\in G^{\al,L}(\T)$ such that
\[
%
%
\normD{w}_{\al,L} \le (1+2 L^\al)\normD{\eps_*}_{\al,L'}
\le (1+2 L^\al) L'^\al \ex^{- \frac{5c}{3} \nu^{-\ga}}.
%
%
\]
Proposition~\ref{propB} now tells us that the annulus map
$\ti F \defeq \Phi^w \circ F_0\col M_1 \righttoleftarrow$ leaves invariant the graph
$\cG \defeq \big\{ \big( \th, \hr + \eps(\th) \big) \mid \th \in \T\big\}$,
with $f\col \T \righttoleftarrow$ as induced dynamics on~$\cG$.
In particular, 
\[
z_1 \defeq \big(\tfrac{1}{2Q}+\Z, \hr+\eps(\tfrac{1}{2Q}) \big) \in
\T\times\R
\]
is a $\sig$-isolated periodic point of period~$q$ for~$\ti F$,
with all its orbit contained in~$\cG$.
Notice that, for~$\nu$ small enough,
$\normD{\eps}_{C^0(\T)} \le \de + \normD{\eps_*}_{C^0(\T)}
\le \nu^2/12 + L'^\al \ex^{-\frac{5c}{3} \nu^{-\ga}}
\le \nu/4$,
hence {$\cG \subset \cV(\hr,\nu/4) \subset \cV(\br,3\nu/4)$ and
$\{ \ti F^s(z_1) \mid s \in \N,\; {2}/{\nu} \le
    s \le {6}/{\nu} \}$ is $2\nu$-dense in $\cV(\br,\nu)$.}


\noinbf{d}
The only shortcoming of the perturbation~$\Phi^w$ is that it is not
supported on the strip $\cV(\br,\nu)$, but this is easy to remedy:
we will multiply~$w$ by a function which vanishes outside
$\cV(\br,\nu)$ without modifying the dynamics in $\cV(\hr,\nu/4)$.
Note that 
\[
\cV(\hr,\nu/4) \subset \cV(\hr,\nu/2) \subset
\cV(\br,\nu).
\]
Let us thus pick $\eta \in G^{\al,L}(\R)$ such that 
$\eta(r) = 1$ for all $r \in [\hr-\nu/4,\hr+\nu/4]$ and 
$\eta(r) = 0$ whenever $\absD{r-\hr} \ge \nu/2$.
According to Lemma~\ref{lembump}, we can achieve
$\normD{\eta}_{\al,L} \le \exp\big( 2^{\ga} c_1 \nu^{-\ga} \big)$
(using, in fact, a non-periodic version of Lemma~\ref{lembump}, with $p=\frac{2}{\nu}$).
We now set 
\[
u(\th,r) \defeq \eta(r) w(\th)
\quad \text{for all $(\th,r) \in \T \times \R$.}
\]
{One can check that~$u$ satisfies conditions~(1) and~(2) of
Proposition~\ref{prop.isolation} for~$\nu$ small enough,}
because then $\normD{w}_{\al,L} \le \demi \ex^{- \frac{3c}{2} \nu^{-\ga}}$,
while $\normD{\eta}_{\al,L} \le \ex^{\frac{c}{2} \nu^{-\ga}}$.
Since $F \defeq \Phi^u\circ F_0$ and~$\ti F$ coincide on~$\cG$ (in
fact on all of $\cV(\hr,\nu/4)$), requirements~(3) and~(4) are also fulfilled.



\section{Coupling lemma and synchronized diffusion \textit{\`a la Herman}}
\label{secSynchrDiffH}


\subsection{Coupling lemma}


The following ``coupling lemma'' due to M.~Herman was already used in
\cite{hms}, \cite{ms}, \cite{lms} to construct examples of unstable
near-integrable Hamiltonian flows.

\label{seclemcoupl}
\begin{lemma}   \label{lemcoupl} 
Let $M$ and~$M'$ be symplectic manifolds.
Suppose we are given two maps, 
$F \col M \righttoleftarrow$ and $G \col M' \righttoleftarrow$,
and two Hamiltonian functions $f \col M\to\R$ and $g \col M'\to\R$
which generate complete vector fields and define time-1 maps~$\Phi^f$ 
and~$\Phi^g$.
Suppose moreover that $z_*\in M$ is $F$-periodic, of period~$q$, and that
\begin{align} \label{condition.f1} 
& & & & &f(z_*)=1, & 
&\dd f(z_*)=0 \\ \label{condition.f2}
& & & & & f(F^s(z_*))=0, & 
& \dd f(F^s(z_*)) =0\quad \text{for $1\le s\le q-1$}. \end{align}
Then $f\otimes g$ generates a complete Hamiltonian vector field
and the maps
\[
T \defeq \Phi^{f\otimes g}\circ(F\times G) \col M \times M' \righttoleftarrow
\quad\text{and}\quad
\psi \defeq \Phi^g\circ G^q \col M' \righttoleftarrow
\]
satisfy
\beglab{eq:orbsynchr}
%
%
T^{n q+s}(z_*,z') = \big( F^s(z_*),G^s\circ \psi^n(z') \big)
\edla
for all $z'\in M'$ and $n, s \in \Z$ such that $0 \le s \le q-1$.
\end{lemma}

We have denoted by~$f\otimes g$ the function $(z,z')\mapsto f(z)g(z')$,
and by~$F\times G$ the map $(z,z')\mapsto (F(z),G(z'))$.

\begin{proof} See \cite{hms}. \end{proof}


%
\subsection{Proof of Proposition~\ref{prop.diffusion.discrete}}
\label{sec:pfPropDiffDiscr}


\noinbf{a}
Given $\br\in\R$ and $\nu>0$ small enough, we apply
Proposition~\ref{prop.isolation} with
\[
\sig \defeq \ex^{-2c\nu^{-\ga}}
\]
(where $c=c(\al,L)$ is provided by Proposition~\ref{prop.isolation})
and an integer $\ell\ge 6/\nu$ that we will specify later.

We thus get a function $u\in G^{\al,L}(M_1)$ and a map $F=\Phi^u\circ
F_0\col M_1 \righttoleftarrow$ satisfying properties (1)--(4) of
Proposition~\ref{prop.isolation}.
We call~$\sz$ the $\sig$-isolated periodic point mentioned in property~(4),
the period of which is an integer $q \in [\ell,3\ell/\nu]$.

Let $f \defeq \eta_{\sz,\sig}$ be defined by
Lemma~\ref{lemmaetaznu}.
Observe that $f$, ${F}$ and $z_*=\sz$ satisfy conditions
\eqref{condition.f1}--\eqref{condition.f2} of Lemma~\ref{lemcoupl}
because~$\sz$ is a $\sig$-isolated periodic point.


\noinbf{b}
We now define $g\col M_2 \to \R$ by the formula $g(r_2,\th_2)= -\frac{1}{2\pi q} \sin(2\pi \th_2)$.
According to~\eqref{eq:trivex}, we have
$\Phi^g(\th_2,r_2)= \big(\th_2, r_2+\frac{1}{q}\cos(2\pi\th_2) \big)$
for all $(\th_2,r_2) \in M_2$. 
In particular,
\[
\Phi^g(0+\Z, r_2) = \big(0+\Z, r_2 + \tfrac{1}{q}\big)
\quad \text{for all $r_2\in \R$}.
\]
On the other hand,~\eqref{eq:defFzGz} gives
$G_0^s(\th_2,r_2) = \big( \th_2 + s(\om_2+r_2) + \Z, r_2 \big)$ for
all $s\in\Z$.
Therefore
\[
\psi \defeq \Phi^g \circ G_0^q
\]
satisfies $\psi^n(0+\Z,-\om_2) = (0+\Z, -\om_2 + \frac{n}{q})$ for all
$n\in\Z$, whence
\beglab{eq:fullorbpsi}
G_0^s \circ \psi^n\big(\ssz{0,0}\big) = \ssz{n,s}
\quad\text{with}\ens
\ssz{n,s} \defeq \Big( \frac{sn}{q} + \Z, -\om_2 + \frac{n}{q} \Big)
\quad \text{for all $n,s\in \Z$.}
\edla


\noinbf{c}
Let $v \defeq f \otimes g$.
We now apply Lemma~\ref{lemcoupl} with the above functions $f$ and~$g$
and the maps $F$ and~$G_0$, taking $z_*=\sz$.
In view of~\eqref{eq:orbsynchr}--\eqref{eq:fullorbpsi}, we get
\beglab{eq:Torb}
T^{nq+s}\big( \sz, \ssz{0,0} \big) =
\big( F^s(\sz), \ssz{n,s} \big)
\quad\text{for all $n, s \in \Z$ such that $0 \le s \le q-1$,}
\edla
with $T = \Phi^v \circ (F\times G_0) = \Phi^v \circ \Phi^u \circ
(F_0\times G_0)$.

Notice that, for~$\nu$ small enough, we have $\sig < \nu/4$, hence
$f\equiv0$ on $\cV(\br,\nu)^c$, thus~$v$ satisfies condition~(1) of
Proposition~\ref{prop.diffusion.discrete}, and we already knew
that~$u$ also did.

We have $\normD{u}_{\al,L} \le \demi \exp(- c \nu^{-\ga})$ with
$c=c(\al,L)$ stemming from Proposition~\ref{prop.isolation}.
We can also achieve $\normD{v}_{\al,L} \le \demi \exp(- c \nu^{-\ga})$
by choosing appropriately~$\ell$.
Indeed, calling~$K$ the Gevrey-$(\al,L)$ norm of the function
$\th\mapsto \frac{1}{2\pi} \sin(2\pi \th)$ and using $q\ge\ell$
and~\eqref{ineqnormetaznu},
we get
$\normD{v}_{\al,L} \le \frac{K}{q} \normD{\eta_{z,\nu}}_{\al,L} 
\le \frac{K}{\ell} \exp(c_2{\sig^{-\ga}})$, 
where $c_2=c_2(\al,L)$ stems from Lemma~\ref{lemmaetaznu}.
Therefore, the condition
\beglab{ineq:condellL}
q \ge \ell \ge L \defeq 2 K \ex^{c \nu^{-\ga}} \ex^{c_2 \sig^{-\ga}}
\edla
is sufficient to ensure that~$u$ and~$v$ satisfy condition~(2) of
Proposition~\ref{prop.diffusion.discrete}.
We will fine-tune our choice of~$\ell$ later, when considering the
``diffusion speed'' of the $T$-orbit described by~\eqref{eq:Torb}.


\noinbf{d}
We note that $\cV(\bar r,\nu) \times M_2$ is invariant by~$T$ because it
is invariant by~$T_0$ as well as by~$\Phi^{tu}$ and~$\Phi^{tv}$ for
all $t\in\R$ in view of the condition on the supports of~$u$ and~$v$.
Let $b\defeq \quart$.
To get condition~(3) of Proposition~\ref{prop.diffusion.discrete} and
thus complete its proof, we will require the following
%

\begin{lemma}    \label{lem:easydense}
The set
$\big\{ \big( F^s(\sz), \ssz{n,s} \big) \mid n,s\in\Z,\; \frac{2}{\nu} \le s
\le \frac{6}{\nu} \big\}$
is $3\nu$-dense in $\cV(\bar r,\nu) \times M_2$
if~\eqref{ineq:condellL} holds and~$\nu$ is small enough.
\end{lemma}


Taking Lemma~\ref{lem:easydense} for granted, we now show how to
choose~$\ell$ so as to make $\cV(\bar r,\nu) \times M_2$ 
$\big(3\nu,\tau,\tau^{b}\big)$-diffusive for~$T$
with $\tau \defeq E_{3c\ga,\ga}(\nu)$.

Given arbitrary $z\in \cV(\bar r,\nu) \times M_2$,
Lemma~\ref{lem:easydense} yields~$n$ and~$s$ integer such that
$\hz \defeq \big( F^s(\sz), \ssz{n,s} \big) = T^{nq+s}\big( \sz, \ssz{0,0} \big)$ is $3\nu$-close
to~$z$.
For any $m\ge1$, comparing the last coordinates of~$\hz$
and $T^{mq}(\hz) = \big( F^s(\sz), \ssz{n+m,s} \big)$,
namely $-\om_2+\frac{n}{q}$ and $-\om_2+\frac{n+m}{q}$, we see that
$d(T^{mq}(\hz), \hz) \ge m/q$, hence
\beglab{ineq:Tqhzell}
d(T^{q^3}(\hz), z) \ge q - 3\nu
\quad\text{with}\ens
\ell \le q \le \frac{3\ell}{\nu}.
\edla

Let $\mu \defeq \ex^{c\nu^{-\ga}} = \sig^{-1/2}$.
The number~$L$ of~\eqref{ineq:condellL} is
$2 K \mu\, \ex^{c_2 \mu^{2\ga}} \le \mu^2 \ex^{c_2 \mu^{2\ga}} 
< \ex^{(c_2 + \frac{1}{\ga})\mu^{2\ga}}$ and
$(c_2 + \frac{1}{\ga})\mu^{2\ga} \le b \mu^{3\ga}$
provided~$\nu$ is small enough, and then
\begla
L < \ex^{b\, \ex^{C \nu^{-\ga}}}
\ens \text{with} \ens
C \defeq 3 c \ga.
\edla
Since $b < \tiers$, we can satisfy~\eqref{ineq:condellL} by choosing $\ell \defeq
\flo{\frac{\nu}{3} \ex^{\tiers \ex^{C \nu^{-\ga}}}}$
and we then have
$q^3 \le (3\ell/\nu)^3 \le \ex^{\ex^{C \nu^{-\ga}}} = E_{C,\ga}(\nu)$.
%

On the other hand, $q-3\nu \ge \ell - 3\nu \ge
\ex^{b\,\ex^{C\nu^{-\ga}}} = E_{C,\ga}(\nu)^b$ for~$\nu$ small enough.
%
%
We thus get property~(3) of Proposition~\ref{prop.diffusion.discrete} from~\eqref{ineq:Tqhzell}.
The proof of that Proposition is thus complete up to the proof of
Lemma~\ref{lem:easydense}.


\subsection{Proof of Lemma~\ref{lem:easydense}}
 

We keep the notations ans assumptions of Section~\ref{sec:pfPropDiffDiscr} and give
ourselves an arbitrary $z = (z_1,z_2) \in \cV(\bar r,\nu) \times M_2$.
We look for integers~$n$ and~$s$ such that $d\big( \big( F^s(\sz), \ssz{n,s} \big), (z_1,z_2) \big)
\le 3\nu$ and $\tfrac{2}{\nu} \le s \le \tfrac{6}{\nu}$.

By property~(4) of Proposition~\ref{prop.isolation}, we can choose the
integer~$s$ so that
\[
d\big( F^s(\sz), z_1 \big) \le 2\nu,
\qquad \tfrac{2}{\nu} \le s \le \tfrac{6}{\nu}.
\]
On the other hand, writing $z_2 = (\th_2+\Z,r_2)$ with $\th_2,r_2 \in
\Z$, we see from~\eqref{eq:fullorbpsi} that the last coordinate
of~$\ssz{n,s}$ will be $\nu$-close to~$r_2$ if and only if $\absD{n-q(\om_2+r_2)}
\le \nu q$.
Let us denote by~$n_*$ the integer nearest to $q(\om_2+r_2)$, thus
$\absD{n_*-q(\om_2+r_2)} \le \dem$.
To conclude, it is sufficient to take~$n$ of the form $n=n_*+m$
with~$m$ integer such that
\beglab{ineq:final}
\absD{m} \le \nu q-\dem
\quad\text{and}\quad
\Dist\bigg( \frac{s(n_*+m)}{q}, \th_2 + \Z \bigg) \le \nu.
\edla
Indeed, we will then have
$d\big( \big( F^s(\sz), \ssz{n,s} \big), (z_1,z_2) \big)
\le \sqrt{4\nu^2 + \nu^2 + \nu^2} < 3\nu$.


The second part of~\eqref{ineq:final} is equivalent to
$\dst \Dist\Big( m, \frac{q\th_2}{s} - n_* + \frac{q}{s} \Z \Big) \le
\frac{\nu q}{s}$.
Let \vspace{-1ex}
\[
I \defeq \Big[ -(\nu q-\dem),\nu q-\dem \Big],
\qquad
J \defeq \Big[ x_2 - \frac{\nu q}{s}, x_2 + \frac{\nu q}{s} \Big]
\quad \text{with} \ens
x_2 = \frac{q\th_2}{s} - n_*.
\]
The whole of~\eqref{ineq:final} is thus equivalent to
\beglab{eq:form}
m \in I 
\quad \text{and} \quad
m \in \frac{kq}{s} + J
\ens\text{for some $k\in\Z$.}
\edla
Now, $\frac{kq}{s} + J \subset I$ is equivalent to 
\beglab{inclusJI}
\absD{I} \ge \absD{J}
\quad \text{and} \quad
-\dem\big(\absD{I} - \absD{J}\big)
\le \frac{kq}{s} + x_2 \le
\dem\big(\absD{I} - \absD{J}\big).
\edla
Since $\De \defeq \absD{I}-\absD{J} = 2\nu q - 1 - 2\nu q/s$,
\eqref{inclusJI} amounts to
$\De\ge0$ and~$k$ belonging to the interval $\big[ -\frac{s x_2}{q}-\frac{s\De}{2q},
-\frac{s x_2}{q}+\frac{s\De}{2q} \big]$,
which has length $s\De/q = 2\nu s - 2\nu - s/q \ge 4-2\nu-6/(\nu L)$
(using $q\ge L$).
That length is $\ge1$ for~$\nu$ small enough and the interval then
contains at least one integer~$k_*$.
Diminishing~$\nu$ if necessary, we then have 
$\absD{\frac{k_* q}{s} + J} = 2\nu q/s \ge \nu^2 L/3 \ge 1$,
hence we can find $m \in (\frac{k_*q}{s} + J)\cap \Z \subset I\cap \Z$,
thus solution to~\eqref{eq:form} or, equivalently, to~\eqref{ineq:final}.



 \section{Continuous time}
 \label{sec.flows}
 

In Sections~\ref{sec:mainmldbr}--\ref{secSynchrDiffH}, we have proved
Theorems~\ref{theorem.onetorus.discrete}
and~\ref{theorem.manytori.discrete},
providing examples of discrete systems of the form $\Phi^v \circ
\Phi^u \circ \Phi^{h_0} \col \T^n\times\R^n \righttoleftarrow$ with diffusive invariant tori
(using Remark~\ref{rem:discretemultidim}).
We will now deduce Theorems~\ref{theorem.onetorus}
and~\ref{theorem.manytori} by a ``suspension'' device adapted from \cite{hms}.


\begin{definition}
  Given an exact symplectic map
  $T \col \T^n\times\R^n \righttoleftarrow$, we call suspension of~$T$
  any non-autonomous Hamiltonian which depends
  $1$-periodically on time
\[ h \col \T^n\times\R^n \times \T \to \R, \]
for which the flow map between the times $t=0$ and $t=1$ exists and coincides with~$T$.
\end{definition}


\begin{lemma}    \label{lem:explicitsusp}
Let $h_0 \col r\in\R^n \mapsto (\om,r) + \dem (r,r)$ as in Section~\ref{sec:HamFlows}.
Suppose that~$u$ and~$v$ are~$C^\infty$ functions on $\T^n\times\R^n$
which generate complete Hamiltonian vector fields.
Let $\psi,\chi\in C^\infty([0,1])$ be such that
\beglab{cond:psichi}
\int_0^1\psi(t)\,\dd t = \int_0^1 \chi(t)\,\dd t = 1, \quad
\supp(\psi) \subset \big[\tiers,\dtiers\big], \quad
\supp(\chi) \subset \big[\dtiers,1\big].
\edla
Then the formula
\beglab{eq:explicitsusp}
h(\th,r,t) \defeq h_0(r) 
+ \psi(t) u\big( \th + (1-t)(\om+r) + \Z^n, r\big)
+ \chi(t) v\big( \th + (1-t)(\om+r) + \Z^n, r\big)
\edla
defines a function on $\T^n\times\R^n\times[0,1]$ which uniquely extends by
periodicity to a $C^\infty$ function on $\T^n\times\R^n\times\T$
and, when viewed as a non-autonomous time-periodic Hamiltonian,
is a suspension of $\Phi^v \circ \Phi^u \circ \Phi^{h_0}$.
\end{lemma}


\begin{proof}
Let $\ph\in C^\infty([0,1])$ have support $\subset [0,\tiers]$ and
$\int_0^1 \ph(t)\,\dd t=1$.
We observe that~\eqref{eq:explicitsusp} entails,
for all $(\th,r,t) \in \T^n\times\R^n\times[0,1]$,
\beglab{eq:varexplicitsusp}
h(\th,r,t) = h_0(r) 
+ \psi(t) u\big( \th + \ti\ph(t)(\om+r) + \Z^n, r\big)
+ \chi(t) v\big( \th + \ti\ph(t)(\om+r) + \Z^n, r\big),
\edla
where $\ti\ph(t) \defeq \int_0^t \big(\ph(t')-1\big)\dd t'$ extends
to a $1$-periodic $C^\infty$ function.
This takes care of the first statement.

Let $z_0\in\T^n\times\R^n$ and let $z(t)=\big(\th(t),r(t)\big)$ denote the maximal solution of the initial
value problem $\tfrac{\dd z}{\dd t} = X_h(z,t)$, $z(0)=z_0$.
Defining $\th^*(t) \defeq \th(t) + \ti\ph(t)\big(\om+r(t)\big)$ and
$z^*(t)\defeq \big(\th^*(t),r(t)\big)$,
we compute
\begin{align*}
&&&&&&&&&&\tfrac{\dd z^*}{\dd t\hspace{.25em}}(t) &= \ph(t) X_{h_0}\big(z^*(t)\big) 
&&&\text{for $t\in\big[0,\tiers\big]$,}&&&&&&&&&&\\[1ex]
&&&&&&&&&& \tfrac{\dd z^*}{\dd t\hspace{.25em}}(t) &= \psi(t) X_{u}\big(z^*(t)\big) 
&&&\text{for $t\in\big[\tiers,\dtiers\big]$,}&&&&&&&&&&\\[1ex]
&&&&&&&&&& \tfrac{\dd z^*}{\dd t\hspace{.25em}}(t) &= \chi(t) X_{v}\big(z^*(t)\big) 
&&&\text{for $t\in\big[\dtiers,1\big]$.}&&&&&&&&&&
\end{align*}
The flow map of~$X_h$ between the times $t=0$ and $t=\tiers$ is thus a
reparametrization of the flow of~$X_{h_0}$:
$t\in \big[0,\tiers\big] \IMP 
z^*(t) = \Phi^{h_0}\big( \int_0^t\ph(t')\,\dd t' \big)$,
whence $z^*\big(\tiers\big) = \Phi^{h_0}(z_0)$ since
$\int_0^{1/3}\ph(t')\,\dd t' = 1$.
Similarly,
$z^*\big(\dtiers\big) = \Phi^{u}\big(z^*\big(\tiers\big)\big)$
since $\int_{1/3}^{2/3}\psi(t')\,\dd t' = 1$,
and $z^*(1) = \Phi^{v}\big(z^*\big(\dtiers\big)\big)$,
since $\int_{2/3}^{1}\chi(t')\,\dd t' = 1$.
We thus get $z^*(1) = \Phi^v \circ \Phi^u \circ \Phi^{h_0}(z_0)$,
which yields the desired result because $z^*(1)=z(1)$.
\end{proof}


Notice that, if $T = \Phi^v \circ \Phi^u \circ \Phi^{h_0}$ satisfies
properties~(1) and~(3) of Theorem~\ref{theorem.onetorus.discrete}, \resp
Theorem~\ref{theorem.manytori.discrete},
then \textit{any suspension of~$T$ of the form~\eqref{eq:varexplicitsusp} satisfies
property~(2) of Theorem~\ref{theorem.onetorus}, \resp
Theorem~\ref{theorem.manytori}}.
The invariance of the torus~$\cT_{(r,s)}$ with $r=0$, \resp
$r\in(X_\eps+\Z)\times\R^{n-1}$, stems from the vanishing of $\pa_\th
h$ and~$\pa_t h$ for $r_1=0$, \resp $r_1\in X_\eps+\Z$.


\begin{lemma}   \label{lem:composSom}
Consider the mapping
\[
\cS_\om \col (\th,r,t) \in \T^n\times \R^n \times \T \mapsto
\big( \th + (1-t)\om +\Z^n, r \big)
\in \T^n\times \R^n.
\]
Let $\al\ge1$ and $\La>0$ be real, and 
$\La_1 \ge \La \Big( 1 + \max\limits_{1\le i \le
  n}\absD{\om_i}^{1/\al} \Big)$.
Then
\[
w\in G^{\al,\La_1}(\T^n\times \R^n)
\IMP
w\circ\cS_\om \in G^{\al,\La}(\T^n\times \R^n\times \T)
\ens\text{and}\ens
\normD{w\circ\cS_\om}_{\al,\La}
\le \normD{w}_{\al,\La_1}.
\]
\end{lemma}

\begin{proof}
One can check that, for every $(p,q,s)\in\N^n\times\N^n\times \N$,
\[
\pa_\th^p \pa_r^q \pa_t^s (w\circ \cS_\om) = 
(-1)^s \sum_{m\in\N^n \;\text{s.t.}\; \absD{m}=s} \om_1^{m_1} \cdots \om_n^{m_n}
(\pa_\th^{p+m} \pa_r^q w)\circ \cS_\om,
\]
whence, with the notation $\Om \defeq \max\limits_{1\le i \le
  n}\absD{\om_i}$,
\begin{align*}
\normD{w\circ \cS_\om}_{\al,\La} &\le \sum_{p,q,m \in \N^n}
\frac{ \Om^{\absD{m}} \La^{\absD{p+q+m}\al} }{ p!^\al q!^\al \absD{m}!^\al }
\normD{\pa_\th^{p+m} \pa_r^q w}_{C^0(\T^n\times\R^n)} \\[1ex]
& = \sum_{\ell,q \in \N^n} A_\ell
\frac{ \La^{\absD{\ell+q}\al }}{ q!^\al }
\normD{\pa_\th^\ell \pa_r^q w}_{C^0(\T^n\times\R^n)}
\quad \text{with}\ens 
A_\ell \defeq \sum\limits_{p,m\in\N^n \;\text{s.t.}\; p+m=\ell}
\frac{ \Om^{\absD{m}} }{ p!^\al \absD{m}!^\al }.
\end{align*}
Now, 
$A_\ell \le \sum\limits_{p+m=\ell} \frac{ \Om^{\absD{m}} }{ p!^\al m!^\al }
\le \Big(\sum\limits_{p+m=\ell} \frac{ \Om^{\absD{m}/\al} }{ p! m! }\Big)^\al
= \frac{1}{\ell!^\al} (1+\Om^{1/\al})^{\absD{\ell}\al}$,
hence
\[
\normD{w\circ \cS_\om}_{\al,\La} 
\le \sum_{\ell,q \in \N^n} 
\frac{ (1+\Om^{1/\al})^{\absD{\ell}\al} \La^{\absD{\ell+q}\al }}{ \ell!^\al q!^\al }
\normD{\pa_\th^\ell \pa_r^q w}_{C^0} 
\le \frac{ \La_1^{\absD{\ell+q}\al }}{ \ell!^\al q!^\al }
\normD{\pa_\th^\ell \pa_r^q w}_{C^0} 
= \normD{w}_{\al,L_1}.
\]
\end{proof}


\begin{lemma}   \label{lem:composRI}
Consider an interval $I\subset[0,1]$ and the mapping
\[
\cR_I \col (\th,r,t) \in \T^n\times \R^n \times I \mapsto
\big( \th + (1-t)r +\Z^n, r , t\big)
\in \T^n\times \R^n \times I.
\]
Let $\al\ge1$ and $\L>0$ be real, and 
$\La \ge L \max\big\{ 2^{3/\al}, 2^{1/\al} L \big\}$.
Then
\[
w\in G^{\al,\La}(\T^n\times \R^n\times I)
\IMP
w\circ\cR_I \in \kG(\T^n\times \R^n\times I)
\ens\text{and}\ens
d_{\al,L}(0,w\circ\cR_I)
\le \normD{w}_{\al,\La}.
\]
\end{lemma}

\begin{proof}
Let $L_j \defeq 2^{-\frac{j-1}{\al}} L$ and $R_j \defeq 2^j$ for each $j\in\N^*$, as in
Appendix~\ref{AppSubsecGev}. 
We also set $\cK_j \defeq \T^n\times \ov B_{R_j} \times I$,
so $\kG(\T^n\times \R^n\times I) = \bigcap_{j\ge1}
G^{\al,L_j}(\cK_j)$.

A simple adaptation of \cite[Remark~A.1]{hms} shows that, 
if $\phi \in G^{\al,L_j}(I)$ and
\beglab{ineq:LaRjLj}
L_j^\al + (R_j + L_j^\al) \normD{\phi}_{\al,L_j,I} - R_j \normD{\phi}_{C^0(I)} 
\le \La^\al,
\edla
then the composition with the mapping
\[
\cR \col (\th,r,t) \in \T^n\times \R^n \times I \mapsto
\big( \th + \phi(t)r +\Z^n, r , t\big)
\in \T^n\times \R^n \times I
\]
has the property
\beglab{ineq:wcircR}
w\in G^{\al,\La}(\cK_j)
\IMP
w\circ\cR \in G^{\al,L_j}(\cK_j)
\ens\text{and}\ens
\normD{w\circ\cR}_{\al,L_j,\cK_j}
\le \normD{w}_{\al,\La}.
\edla
Taking $\phi(t) \defeq 1-t$, since our interval~$I$ is $\subset
[0,1]$, we have
$\normD{\phi}_{C^0(I)} \le 1$ and
$\normD{\phi}_{\al,L_j,I} = \normD{\phi}_{C^0(I)} + L_j^\al$,
hence the \lhs\ of~\eqref{ineq:LaRjLj} equals
\[
L_j^\al + R_j L_j^\al + L_j^\al \big( \normD{\phi}_{C^0(I)} + L_j^\al \big)
\le L_j^\al ( L_j^\al + R_j + 2) 
\le L^{2\al} + 4 L^{\al} \le \dem \La^\al + \dem \La^\al
\]
and~\eqref{ineq:wcircR} allows us to conclude, in view
of~\eqref{eq:defdalL}.
\end{proof}


\subsubsection*{Proof of Theorems~\ref{theorem.onetorus}
and~\ref{theorem.manytori}}


In both cases, we are given $\al>1$, $L>0$ and $\eps>0$.
Let
\[
\La_1 \defeq  \La \Big( 1 + \max\limits_{1\le i \le
  n}\absD{\om_i}^{1/\al} \Big),
\quad
\La \defeq L \max\big\{ 2^{3/\al}, 2^{1/\al} L \big\}.
\]
Since $\al>1$, we can pick $\psi,\chi \in G^{\al,\La}(\T)$ satisfying~\eqref{cond:psichi}.

Let us apply the multidimensional version of
Theorem~\ref{theorem.onetorus.discrete} or
Theorem~\ref{theorem.manytori.discrete} (\cf
Remark~\ref{rem:discretemultidim}) with parameters~$\La_1$ instead
of~$L$
and
\[
\eps_1 \defeq \frac{\eps}{ \max\big\{1,
\normD{\psi}_{\al,\La}, \normD{\chi}_{\al,\La} \big\} }
\]
instead of~$\eps$.
We thus get $u,v\in G^{\al,\La_1}(\T^n\times\R^n)$ and, in the second
case, $X_{\eps_1} \subset [0,1]$,
such that $\normD{u}_{\al,\La_1} + \normD{v}_{\al,\La_1} < \eps_1$
and any suspension of $\Phi^v \circ \Phi^u \circ \Phi^{h_0}$ satisfies
property~(2) of Theorem~\ref{theorem.onetorus}, \resp
Theorem~\ref{theorem.manytori}.

By Lemma~\ref{lem:explicitsusp}, we can choose the suspension to be
\[
h \defeq h_0 + \ti u \circ \cR_{[\frac{1}{3},\frac{2}{3}]}
+
\ti v \circ \cR_{[\frac{2}{3},1]}
\vspace{-1.5ex}
\]
with
\[ \ti u(\th,r,t) \defeq \psi(t) (u \circ \cS_\om)(\th,r,t),
\quad
\ti v(\th,r,t) \defeq \chi(t) (v \circ \cS_\om)(\th,r,t).
\]
Lemma~\ref{lem:composSom} and~\eqref{ineqGevBanAlg} give
\[
\normD{\ti u}_{\al,\La} \le \normD{\psi}_{\al,\La} \normD{u}_{\al,\La_1}, 
\quad
\normD{\ti v}_{\al,\La} \le \normD{\chi}_{\al,\La} \normD{v}_{\al,\La_1}, 
\]
whence $\normD{\ti u}_{\al,\La} + \normD{\ti v}_{\al,\La} < \eps$.
Then, since the distance $d_{\al,L}$ is translation-invariant,
\[
d_{\al,L}(h_0,h) \le 
d_{\al,L}\big(0, \ti u \circ \cR_{[\frac{1}{3},\frac{2}{3}]}\big)
+ d_{\al,L}\big(0, \ti v \circ \cR_{[\frac{2}{3},1]} \big)
\le \normD{\ti u}_{\al,\La} + \normD{\ti v}_{\al,\La}
\]
by Lemma~\ref{lem:composRI} and we are done.


\begin{appendices}

\section{Gevrey estimates}\label{App:Gevrey}
\setcounter{thm}{-1}

We fix real numbers $\al\ge1$ and $L>0$. 


\subsection{Gevrey functions and Gevrey maps}   \label{AppSubsecGev}


Here we adapt definitions and facts taken from \cite{hms}, \cite{ms} and \cite{FMS}.


\subsubsection*{The Banach algebra $G^{\al,L}(\R^M\times K)$ of uniformly Gevrey-$(\al,L)$ functions}


Let $N\ge1$ be integer. We will deal with real functions of~$N$ variables
defined on $\R^M\times K$, where $M\ge0$ and $K\subset \R^{N-M}$ is a Cartesian
product of closed Euclidean balls and tori. 
%
We define the uniformly Gevrey-$(\al,L)$ functions on $\R^M\times K$ by
\begin{multline}
G^{\al,L}(\R^M\times K) \defeq \{ f\in C^\infty(\R^M\times K) \mid \normD{f}_{\al,L} <\infty \},
\\
\label{eq:defGalL}
\normD{f}_{\al,L} \defeq \sum_{\ell\in\N^{N}}
\frac{L^{\absD{\ell}\al}}{\ell !^\al} \normD{\pa^\ell f}_{C^0(\R^M\times K)}.
\end{multline}
We have used the standard notations
$\absD{\ell} = \ell_1+\cdots+\ell_{N}$, $\ell! = \ell_1!\ldots\ell_{N}!$,
$\pa^\ell = \pa_{x_1}^{\ell_1}\ldots\pa_{x_N}^{\ell_{N}}$, 
and
%
%
$ \N \defeq \{0,1,2,\ldots\}$.
%
%
The space $G^{\al,L}(\R^M\times K)$ turns out to be a Banach algebra,
with
\beglab{ineqGevBanAlg}
\normD{fg}_{\al,L} \le \normD{f}_{\al,L} \normD{g}_{\al,L}
\edla
for all~$f$ and~$g$,
%
and there are ``Cauchy-Gevrey inequalities'':
if $0 < L_0 < L$,
then 
%
\begin{equation}	\label{ineqGevCauch}
\sum_{m\in \N^N;\ |m|=p} \normD{\pa^m f}_{\al,L_0} \le 
\frac{p!^\al}{(L-L_0)^{p\al}} \normD{f}_{\al,L}
\quad \text{for all $p\in\N$.}
\end{equation}
%

When necessary, we use the notation $\normD{\,.\,}_{\al,L,\R^M\times K}$
instead of $\normD{\,.\,}_{\al,L}$ to keep track of the domain to
which the norm relates.


\subsubsection*{The metric space $\kG(\R^M\times K)$}


When $M\ge1$, instead of restricting ourselves to uniformly
Gevrey-$(\al,L)$ functions on $\R^M\times K$, we may cover the
factor~$\R^M$ by an increasing sequence of closed balls and consider a
Fr\'echet space accordingly.
For technical reasons, we choose the sequences
\[
L_j \defeq 2^{-\frac{j-1}{\al}} L, \quad
R_j \defeq 2^j
\qquad \text{for $j\in\N^*$,}
\vspace{-1.5ex}
\]
and set
\beglab{eq:defdalL}
\kG(\R^M\times K) \defeq \bigcap_{j\ge1} G^{\al,L_j}\big(\ov B_{R_j}\times K\big),
\quad
d_{\al,L}(f,g) \defeq \sum_{j\ge1} 2^{-j}
\min\big\{ 1,
\normD{g-f}_{\al,L_j,\ov B_{R_j}\times K}
\big\}.
\edla

Clearly, $G^{\al,L}(\R^M\times K) \subset \kG(\R^M\times K)$ but the
inclusion is strict,
and the larger space is a complete metric space for the
distance~$d_{\al,L}$.

This construction is needed in Section~\ref{sec.flows} only.
In the rest of this appendix, we focus on uniformly Gevrey functions and
maps on~$\R^N$ (with $M=N$ and no factor~$K$).


\subsubsection*{Composition with uniformly Gevrey-$(\al,L)$ maps}
For $N\ge1$ integer, we define
\begin{multline}
G^{\al,L}(\R^{N},\R^{N}) \defeq 
\{ F
\in C^\infty(\R^N,\R^N) \mid \normD{F}_{\al,L} <\infty \},
\\
\label{eq:defGalLB}
\normD{F}_{\al,L} \defeq \normD{F\cc1}_{\al,L} + \cdots + \normD{F\cc N}_{\al,L} .
\end{multline}
This is a Banach space.





We also define
\[ 
\cN^*_{\al,L}(f) \defeq 
\sum_{\ell\in\N^N \setminus \{0\}} \frac{L^{\absD{\ell}\al}}{\ell!^\al} 
\normD{\pa^\ell f}_{C^0(\R^N)},
\] 
so that
$\normD{f}_{\al,L} = \normD{f}_{C^0(\R^N)} + \cN^*_{\al,L}(f)$.


\begin{lemma}  \label{lemGevCompos}
Let $L_0\in(0,L)$. There exists $\epsc = \epsc(\al,L,L_0,N)$ such that,
for any $f\in G^{\al,L}(\R^N)$
and $F=(F\cc1,\ldots,F\cc N)\in G^{\al,L_0}(\R^N,\R^N)$, if
\[ 
\cN^*_{\al,L_0}(F\cc1), \ldots, \cN^*_{\al,L_0}(F\cc N) \le \epsc,
\]
then $f\circ(\Id+F) \in G^{\al,L_0}(\R^N)$ and 
$\normD{f\circ(\Id+F)}_{\al,L_0} \le \normD{f}_{\al,L}$.
\end{lemma}


\noindent The proof is in Appendix~A of \cite{FMS}.


\subsection{Comparison estimates for Gevrey flows}


In Section~\ref{secProofNi}, we use comparison estimates for the flows
of two nearby Gevrey Hamiltonian systems.
We prove them here, building upon some facts which are proved in \cite{FMS}
about the flows of Gevrey vector fields.


\begin{lemma}[General case]   \label{lemGenFlow}
Suppose that $0<L_0<L$ and $N\ge1$.
Then there exists $\epsf=\epsf(\al,L,L_0,N)>0$ such that, for every vector field $X \in
G^{\al,L}(\R^N,\R^N)$ with $\normD{X}_{\al,L} \le \epsf$,
%
%
the time-$1$ map $\Phi$ of the flow generated by~$X$ satisfies 
%
\begla
\normD{\Phi - \Id}_{\al,L_0} \le \normD{X}_{\al,L}
\edla
and, if we are given another vector field $\ti X \in
G^{\al,L}(\R^N,\R^N)$ with $\normD{\ti X}_{\al,L} \le \epsf$,
then its time-$1$ map~$\ti\Phi$ satisfies
\beglab{ineqFlows}
\normD{\ti\Phi-\Phi}_{\al,L_0} \le 2 \normD{\ti X-X}_{\al,L}.
\edla
\end{lemma}

\begin{proof}
The first part of the statement is exactly Part~(i) of Lemma~A.1 from
\cite{FMS}.
There, the flow $t\in[0,1] \mapsto \Phi(t)$ was obtained by considering
the functional $\xi \mapsto \cF(\xi)$ defined by
\[
\cF(\xi)(t) \defeq \int_0^t X\circ\big(\Id +
\xi(\tau)\big)\,\dd\tau.
\]
Using an auxiliary $L'\in(L_0,L)$ and Lemma~\ref{lemGevCompos}, it was
shown that, if $\normD{X}_{\al,L}\le\epsf$ small enough, then~$\cF$
maps into itself
\[
\cB \defeq \{\, \xi \in C^0\big( [0,1], G^{\al,L}(\R^N,\R^N) \big)
\mid \normD{\xi} \le \normD{X}_{\al,L} \,\}
\]
(which is a closed ball in a Banach
space) and has a unique fixed point, none other than
$\xi^*(t) \defeq \Phi(t)-\Id$.

In that proof, $\cF$ was shown to be $K$-Lipschitz, with
$K \defeq \max_{i,j} \normD{\pa_{x_j}X\cc i}_{\al,L'}$.
We can ensure $K\le\dem$ by diminishing~$\epsf$ if necessary and using~\eqref{ineqGevCauch}.
Then, for \textit{any} $\xi_0 \in \cB$, the fixed point~$\xi^*$ is the
limit of the sequence of iterates $(\cF^k(\xi_0))_{k\in\N}$ and
$\normD{\xi^*-\xi_0} \le 2 \normD{\cF(\xi_0)-\xi_0}$.

Now, suppose we also have $\normD{\ti X}_{\al,L} \le\epsf$.
The time-$t$ map of~$\ti X$ is thus $\ti\Phi(t) = \Id+\ti\xi^*(t)$, 
with~$\ti\xi^*$ fixed point of $\ti\cF \col \cB \righttoleftarrow$.
Lemma~\ref{lemGevCompos} yields
\[
\normD{\ti\cF(\xi) - \cF(\xi)} = \normD{
\int_0^t (\ti X-X)\circ\big(\Id + \xi(\tau)\big)\,\dd\tau
} \le \normD{\ti X-X}_{\al,L}
\quad\text{for any $\xi\in\cB$,}
\]
thus we can compare the fixed points~$\xi^*$ and~$\ti\xi^*$ by writing
the former as the limit of the sequence $(\cF^k(\xi_0))_{k\in\N}$ with
$\xi_0 \defeq \ti\xi^*$; we get
\[
\normD{\xi^* - \ti\xi^*} \le 2 \normD{\cF(\ti\xi^*)-\ti\xi^*}
= 2 \normD{\cF(\ti\xi^*)-\ti\cF(\ti\xi^*)} \le \normD{\ti X-X}_{\al,L},
\]
which yields the desired result.
\end{proof}


\begin{lemma}[Hamiltonian case]   \label{lemHamFlow}
Suppose that $0<L_0<L$ and $n\ge1$.
Then there exist $\epsH,C_0>0$ such that, for every $u\in
G^{\al,L}(\R^{2n})$ with $\normD{u}_{\al,L} \le \epsH$,
\begla
\normD{\Phi^u - \Id}_{\al,L_0} \le C_0 \normD{u}_{\al,L},
\edla
and, given another $\ti u \in
G^{\al,L}(\R^{2n})$ with $\normD{\ti u}_{\al,L} \le \epsH$,
\beglab{ineqHamFlows}
\normD{\Phi^{\ti u}-\Phi^u}_{\al,L_0} \le C_0 \normD{\ti u-u}_{\al,L}.
%
%
\edla
\end{lemma}

\begin{proof}
Let $L' \defeq (L_0+L')/2$.
Any $u \in G^{\al,L}(\R^{2n})$ generates a Hamiltonian vector
field~$X_u$ which, according to~\eqref{ineqGevCauch} with $p=1$,
satisfies
\[
\normD{X_u}_{\al,L'} =
\sum_{m\in \N^{2n};\ |m|=1} \normD{\pa^m u}_{\al,L'} \le 
(L-L')^{-\al} \normD{u}_{\al,L}.
\]
Similarly, $\normD{X_{\ti u} - X_u}_{\al,L'} \le (L-L')^{-\al}
\normD{\ti u - u}_{\al,L}$.
Thus, with $\epsH \defeq (L-L')^{\al} \epsf(\al,L',L_0,2n)$ and
$C_0 \defeq 2 (L-L')^{-\al}$,
%
%
we get
\[
\normD{u}_{\al,L}, \normD{\ti u}_{\al,L} \le \epsH
\IMP
\normD{\Phi^u - \Id}_{\al,L_0} \le \dem C_0 \normD{
u}_{\al,L}
\ens\text{and}\ens
\normD{\Phi^{\ti u} - \Phi^u}_{\al,L_0} \le C_0 \normD{\ti u - u}_{\al,L}.
\]
%

%
\end{proof}


\begin{cor}[Iteration of maps of the form $\Phi^v \circ \Phi^u \circ T_0$]   \label{lemCompFlows}
Suppose that $n\ge1$.
Then there exist $\epsd,C_1>0$ such that, for every $u,v,\ti u,\ti v\in
G^{\al,L}(\R^{2n})$ such that 
\beglab{ineqassumptuv}
\normD{u}_{\al,L} + \normD{v}_{\al,L} \le \epsd, \quad
\normD{\ti u}_{\al,L} + \normD{\ti v}_{\al,L} \le \epsd
\edla
and for every $z\in\R^{2n}$, 
the orbits of~$z$ under the maps 
$T\defeq \Phi^v \circ \Phi^u \circ T_0$ and
$\ti T\defeq \Phi^{\ti v} \circ \Phi^{\ti u} \circ T_0$ satisfy
\beglab{ineqDistortEstim}
\Dist\big(T^k(z),\ti T^k(z) \big) \le 3^k C_1 \big(
\normD{\ti u-u}_{\al,L} + \normD{\ti v-v}_{\al,L}
\big) 
\quad \text{for all $k \in \N$.}
\edla
\end{cor}

\begin{proof}
For any $u, v, \ti u, \ti v \in G^{\al,L}(\R^{2n})$ and $z,\ti z \in
\R^{2n}$, the maps
$T \defeq \Phi^v \circ \Phi^u \circ T_0$ and
$\ti T\defeq \Phi^{\ti v} \circ \Phi^{\ti u} \circ T_0$ satisfy
\begin{align*}
\Dist\big(T(z),\ti T(z)\big) & \le
\Dist\big( 
\Phi^{v} ( \Phi^{u} ( T_0 (z) )), \Phi^{v} ( \Phi^{u} ( T_0 (\ti z) ))
\big)
\\[1ex] & \quad + \Dist\big( 
\Phi^{v} ( \Phi^{u} ( T_0 (\ti z) )), \Phi^{v} ( \Phi^{\ti u} ( T_0 (\ti z) ))
\big)
\\[1ex] & \quad \quad + \Dist\big( 
\Phi^{v} ( \Phi^{\ti u} ( T_0 (\ti z) )), \Phi^{\ti v} ( \Phi^{\ti u} ( T_0 (\ti z) ))
\big)
\\[1ex] & \hspace{-5em} \le 
(\LIP\Phi^v)(\LIP\Phi^u)(\LIP T_0) \Dist(\ti z,z)
+ (\LIP\Phi^v) \normD{\Phi^{\ti u}-\Phi^u}_{C^0(\R^{2n})}
+ \normD{\Phi^{\ti v}-\Phi^v}_{C^0(\R^{2n})}.
\end{align*}
%
%
On the one hand, $\LIP T_0 = 2$.
On the other hand, for any $L_0>0$, the Lipschitz constant of a
map~$\Psi$ such that $\Psi-\ID \in G^{\al,L_0}(\R^{2n},\R^{2n})$ is bounded by
%
$1 + \LIP(\Psi-\ID) \le 1 + L_0^{-\al} \normD{\Psi-\ID}_{\al,L_0}$
(using the mean value inequality, \eqref{eq:defGalL}
and~\eqref{eq:defGalLB}).
Applying Lemma~\ref{lemHamFlow} with $L_0 = L/2$,
we can thus choose~$\epsd$ so that assumption~\eqref{ineqassumptuv} entails
\begin{gather*}
\LIP\Phi^u, \LIP\Phi^v \le 1+2^\al L^{-\al} C_0 \epsd \le (3/2)^{1/2}
\\[-1ex] \intertext{and} 
\normD{\Phi^{\ti u}-\Phi^u}_{\al,L_0} \le C_0 \normD{\ti u-u}_{\al,L},
\quad
\normD{\Phi^{\ti v}-\Phi^v}_{\al,L_0} \le C_0 \normD{\ti v-v}_{\al,L},
\end{gather*}
whence
$\Dist\big(T(z),\ti T(z)\big) \le 3 \Dist(\ti z,z) + \eta$
with $\eta \defeq (3/2)^{1/2} C_0 \big( \normD{\ti u-u}_{\al,L}
+ \normD{\ti v-v}_{\al,L} \big)$.
Iterating this, we get
$\Dist\big(T^k(z),\ti T^k(z)\big) \le 3^k ( \Dist(\ti z,z) + \dem\eta)
- \dem\eta$ for all $k\in\N$,
thus we can conclude by choosing $C_1 \defeq \dem (3/2)^{1/2} C_0$.
\end{proof}


\subsection{A Gevrey inversion result}   \label{AppSubsecInv}
In Section~\ref{secUseGevInv},
 we use the following
\begin{lemma}   \label{lemGevInv}
Suppose $L<L_1$. Then there exists $\epsi = \epsi(\al,L,L_1)$ such that, for
every $\eps \in G^{\al,L_1}(\R)$, if $\normD{\eps}_{\al,L_1} \le
\epsi$, then
$\ID+\eps$ is a diffeomorphism of~$\R$ and
\begin{align}
%
%
& (\ID+\eps)\ii = \ID+\ti\eps 
\quad\text{with}\quad
\normD{\ti\eps}_{\al,L} \le \normD{\eps}_{\al,L_1}, \\[1ex]
\label{ineqGevInv}
& \normD*{g\circ(\ID+\eps)\ii}_{\al,L} \le \normD{g}_{\al,L_1}
\quad \text{for any $g\in G^{\al,L_1}(\R)$.}
\end{align}
\end{lemma}


\begin{proof}
Let $L' \defeq (L+L')/2$.
%
%
We use Lemma~\ref{lemGevCompos} and define
\[
\epsi \defeq \min\big\{ \dem(L_1-L')^\al, \epsc(\al,L',L,1), \epsc(\al,L_1,L,1) \big\}.
\]
%
Given $\eps \in G^{\al,L_1}(\R)$ such that
$\normD{\eps}_{\al,L_1} \le \epsi$,
the functional
\[
\cF \col f \in \cB \mapsto -\eps\circ(\ID+f),
\quad\text{where 
$\cB \defeq \{ f \in G^{\al,L}(\R) \mid \normD{f}_{\al,L}
\le\normD{\eps}_{\al,L_1} \}$,}
\]
is well defined (because $\normD{\eps}_{\al,L_1} \le
\epsc(\al,L',L,1)$ and $\eps\in G^{\al,L'}(\R)$), 
maps~$\cB$ into itself (we even have $\normD{\cF(f)} \le \normD{\eps}_{\al,L'}$), 
and is $K$-Lipschitz with
$K \defeq \normD{\eps'}_{\al,L'}$
(using also~\eqref{ineqGevBanAlg} and the mean value inequality).
But~\eqref{ineqGevCauch} yields 
$\normD{\eps'}_{\al,L'} \le 
(L_1-L')^{-\al} \normD{\eps}_{\al,L_1}
\le \dem$, 
which implies that~$\cF$ is a contraction, and also that $\ID+\eps$ is
a diffeomorphism of~$\R$ (since its derivative stays $\ge 1/2$).
The unique fixed point~$\ti\eps$ of~$\cF$ in~$\cB$ is
$(\ID+\eps)\ii-\ID$,
which yields $\normD{\ti\eps}_{\al,L} \le \normD{\eps}_{\al,L_1}
\le \epsc(\al,L_1,L,1)$
and hence~\eqref{ineqGevInv} by another application of
Lemma~\ref{lemGevCompos}.

\end{proof}


\subsection{Gevrey functions with small support}   \label{AppSubsecBump}
From now on we suppose $\al>1$.
We quote without proof Lemma~3.3 of \cite{ms}:


\begin{lemma}   \label{lembump}
There exists a real $c_1 = c_1(\al,L)>0$ such that, for each real $p>2$, the
space $G^{\al,L}(\T)$ contains a function $\eta_p$ which takes its
values in $[0,1]$ and satisfies
\[
-\frac{1}{2p} \le \th \le \frac{1}{2p}
\ens \Rightarrow \ens 
\eta_p(\th+\Z) = 1, \qquad
\frac{1}{p} \le \th \le 1 - \frac{1}{p}
\ens \Rightarrow \ens 
\eta_p(\th+\Z)=0
\]
and
\beglab{ineqnormetap}
\normD{\eta_p}_{\al,L} \le \exp\big(c_1 \, p^{\frac{1}{\al-1}} \big).
\edla
\end{lemma}


\smallskip

The proof can be found in \cite[p.~1633]{ms}.
This easily implies
\begin{lemma}  \label{lemmaetaznu}  
%
%
There exists a real $c_2=c_2(\al,L)>0$ such that,
for any $z \in \T\times\R$ and $\nu>0$, there is a function
$\eta_{z,\nu} \in G^{\al,L}(\T\times\R)$ which takes its
values in $[0,1]$ and satisfies 
\begin{gather}
\notag \text{$\eta_{z,\nu}\equiv 1$ on $B(z,\nu/2)$,} \qquad
\text{$\eta_{z,\nu}\equiv 0$ on $B(z,\nu)^c$}
\\[-2ex]
\intertext{and}
\label{ineqnormetaznu}
\normD{\eta_{z,\nu}}_{\al,L} \le \exp(c_2{\nu^{-\frac{1}{\al-1}}}).
\end{gather}
\end{lemma}

Here, for arbitrary $\ti\nu>0$, we have denoted by $B(z,\ti\nu)$ the 
closed 
ball relative to $\normD{\,.\,}_\infty$ centred at~$z$ with
radius~$\ti\nu$.




\section{Some estimates on doubly exponentially growing sequences}


According to~\eqref{eqdefNi}, the increasing sequence~$(N_i)_{i\ge1}$ is defined by
\[
N_1 \defeq \ceil*{\exp( {4\ka/\eps} )},
\qquad
N_i \defeq N_{i-1} \ceil*{ \exp \big( \exp \big( 
\ti C ( N_{i-1} \ln N_{i-1} )^\ga 
\big) \big) }
\quad\text{for $i\ge2$,}
\]
where $0<\eps\le 1$,
$\ka\ge1$
and $\ti C \defeq \max\{6c\ga,1/\ga\}$, with $c,\ga>0$.
Here, we show a few inequalities which are used in
Section~\ref{secProofNi}.
Recall that $\nu_i \defeq \frac{1}{N_i\ln N_i}$ and $\xi_i \defeq
\ex^{-c\nu_i^{-\ga}}$. 


\begin{lemma}    \label{lemineqNi}
One has \vspace{-.75ex}
\begin{align} 
& & & & & & & &
\label{ineqlnNiqi}
\ln N_i & \ge 4^i \ka / \eps 
& &\text{for every $i\ge 1$,} & & &\\[1ex]
& & & & & & & &
\label{ineqgeomNixi}
N_{i+1} \xi_{i+1} & \le \dem N_i \xi_i
& &\text{for $i$ large enough,} & & &\\[1ex]
& & & & & & & &
\label{ineqexpNixi}
N_{i+1} \xi_{i+1} & \le 3^{-E_{3c\ga,\ga}(\nu_i)}
& &\text{for $i$ large enough.} & & &
\end{align}
\end{lemma}

\begin{proof}
We have $\ln(N_1) \ge 4\ka/\eps$ and, 
by virtue of~\eqref{ineqexp},
$N_1 \ge 4\ka/\eps \ge 4$. 
Now, for $i\ge2$, since $\ga\ti C \ge 1$,
we have 
\[
\ln N_i \ge \exp \big( \ti C ( N_{i-1} \ln N_{i-1} )^\ga \big)
= \Big[ \exp \big( \ga\ti C ( N_{i-1} \ln N_{i-1} )^\ga \big) \Big]^{1/\ga}
\ge \Big[ \exp ( N_{i-1} \ln N_{i-1} )^\ga \Big]^{1/\ga}
\]
and~\eqref{ineqexp} yields
$\ln(N_i) \ge N_{i-1} \ln N_{i-1} \ge 4 \ln N_{i-1}$, whence~\eqref{ineqlnNiqi} follows.

\ssk

We have $\ln\frac{1}{N_i\xi_i} = c (N_i\ln N_i)^\ga - \ln N_i$ and, since
$\ln(N_i) \ll (N_i \ln N_i)^\ga$,
\[
c \La_i^\ga \ge \ln\frac{1}{N_i\xi_i}  \ge c (\La_i/\sqrt3)^\ga
\quad\text{for~$i$ large enough,}
\quad \text{where $\La_i \defeq N_i \ln N_i = 1/\nu_i$.}
\]
Inequality~\eqref{ineqgeomNixi},
 being equivalent to
\[
\ln \frac{1}{N_{i+1} \xi_{i+1}} \ge 
\ln \frac{1}{N_i \xi_i} + \ln 2
\quad\text{for~$i$ large enough,}
\]
thus results from $(\La_{i+1}/\sqrt3)^\ga \ge \La_i^\ga + \frac{\ln 2}{c}$
(which holds for~$i$ large enough because $N_{i+1} \ge 3 N_i$, hence
$\La_{i+1} = N_{i+1} \ln N_{i+1} > 3 \La_i$).

Let $C \defeq 3c\ga$. Inequality~\eqref{ineqexpNixi}, being equivalent to
\[
\ln\frac{1}{N_{i+1} \xi_{i+1}} \ge (\ln3) E_{C,\ga}(1/\La_i)
\quad\text{for~$i$ large enough,}
\]
results from $\La_{i+1}^\ga \ge \frac{3^{\ga/2}\ln3}{c} E_{C,\ga}(1/\La_i)$,
which holds since $E_{C,\ga}(1/\La_i) = \ceil*{ \exp \big( \exp ( 
C \La_i^\ga 
) \big) }$ and
\[
\La_{i+1}^\ga =
N_i^\ga (\ln N_{i+1})^\ga \ceil*{ \exp \big( \ga \exp ( 
\ti C \La_i^\ga 
) \big) },
\quad N_i^\ga (\ln N_{i+1})^\ga \ge \tfrac{3^{\ga/2}\ln3}{c}
\]
and 
$\ga \exp(\ti C \La_i^\ga) \ge \exp(C \La_i^\ga)$
for~$i$ large enough since $\ti C > C$.
\end{proof}


\end{appendices}

\newpage

\noindent 
\textbf{Acknowledgements.}
Both authors thank CNRS UMI3483 -- Fibonacci Laboratory and the Centro
Di Ricerca Matematica Ennio De Giorgi in Pisa for their hospitality. The first author is supported by  ANR-15-CE40-0001.


\vspace{.4cm}



\vspace{.6cm}

\noindent
Bassam Fayad\\ 
CNRS UMR7586 -- IMJ PRG \\
email:\,{\tt{bassam.fayad@imj-prg.fr}}


\vspace{.4cm}

\noindent
David Sauzin\\ 
CNRS UMR 8028 -- IMCCE \\
Observatoire de Paris, PSL University\\
email:\,{\tt{david.sauzin@obspm.fr}}



\begin{thebibliography}{McDS95}

 




\bibitem[Bi66]{Bi} G. D. Birkhoff. 
\newblock {\it Dynamical systems.} With an addendum by Jurgen Moser.
\newblock American Mathematical Society Colloquium Publications, Vol. IX. American
Mathematical Society, Providence, R.I., 1966.

 \bibitem[BFN15]{BFN_point} A.~Bounemoura, B.~Fayad, L.~Niederman, {\it Double exponential stability for
generic real-analytic elliptic equilibrium points}, 51 p. 	arXiv:1509.00285.


 \bibitem[BFN17]{BFN} A.~Bounemoura, B.~Fayad, L.~Niederman,
%
\newblock
 {Superexponential stability of quasi-periodic motion in
              {H}amiltonian systems},
\newblock
 {\it Comm. Math. Phys.} \textbf{350} (2017), no.~1,
 {361--386}.


\bibitem[FF18]{FF} G. Farr\'e and B. Fayad, {\it Instability around
    quasi-periodic motion}, preprint 2018.


\bibitem[Fa18]{F} B. Fayad, {\it Lyapunov unstable elliptic equilibria},
  preprint 2018, 13~p. [arXiv:1809.09059]


\bibitem[FK18]{FKicm}  B. Fayad and R. Krikorian, {\it Some questions around quasi-periodic dynamics.} Proc. Int. Cong. of Math. -- 2018
Rio de Janeiro, Vol. 2, p. 1905--1928.
%


\bibitem[FMS18]{FMS} B.~Fayad, J.-P.~Marco and D.~Sauzin,
%
{\it Attracted by an elliptic fixed point},
preprint 2018, 19~p. [arXiv:1712.03001]


\bibitem[He84]{Hasterisque} M.~R.~Herman, {\it Sur les courbes
    invariantes par les diff\'eomorphismes de l'anneau, Volume~1}, 
%
%
With an appendix by A.~Fathi.
Ast\'erisque {\bf 103} (1983), 221~p.


\bibitem[LMS18]{lms} L. Lazzarini, J.-P. Marco and D. Sauzin,
Measure and capacity of wandering domains in Gevrey near-integrable exact symplectic systems,
%
%
{\it Memoirs of the Amer. Math. Soc.} \textbf{257}, no. 1235
  (in press, expected Dec.~2018), 106 pp.


\bibitem[MS03]{hms} J.-P.~Marco and D. Sauzin,
{Stability and instability for Gevrey quasi-convex near-integrable Hamiltonian
systems,} 
{\it Publ. Math. I.H.E.S.} {\bf 96} (2003), 199--275. 


\bibitem[MS04]{ms} J.-P.~Marco and D. Sauzin,
{Wandering domains and random walks in Gevrey near-integrable Hamiltonian systems,}
{\it Ergodic Theory \& Dynam. Systems} {\bf  24} (2004), no.~5, 1619--1666.


\bibitem[MG95]{MG} A. Morbidelli and A. Giorgilli, 
%
\newblock
    Superexponential stability of {KAM} tori,
%
\newblock
{\it J. Stat. Phys.} \textbf{78} (1995), 1607--1617.


\bibitem[PW94]{Wiggins} A. D. Perry, S. Wiggins.  
%
  \newblock {KAM tori are very sticky: Rigorous lower bounds on the
    time to move away from an invariant Lagrangian torus with linear
    flow}.
%
  \newblock {\it Physica D} \textbf{71} (1994), 102--121.


\bibitem[Po00]{Popov} 
\newblock
G. Popov,
\newblock
 {Invariant tori, effective stability, and quasimodes with
              exponentially small error terms. {I}. {B}irkhoff normal forms},
\newblock
 {\it Ann. Henri Poincar\'{e}},
\textbf{1} (2000), no.~2,
\newblock
 {223--248}.


\bibitem[SM71]{SM} C.~L. Siegel, J.~ Moser,
\newblock 
{\it Lectures on celestial mechanics},
\newblock {Springer Verlag}, 1971.


\end{thebibliography}
\end{document}